\documentclass[11pt]{article}
\usepackage{amssymb}
\usepackage{amsmath,amsthm}
\usepackage{color}
\usepackage{graphicx}
\usepackage{amsfonts}

\def\lf{\lfloor}
\def\rf{\rfloor}

\def\Z{{\mathbb Z}}

\def\La {{\Lambda}}
\def\si {{\sigma}}
\def\la {{\lambda}}
\def\ga{{ \gamma}}
\def\Ga{{ \Gamma}}
\def\eps{{ \epsilon}}

\newcommand{\nn}{\nonumber}
\newcommand{\dis}{\displaystyle}

\newtheorem{thm}{Theorem}
\newtheorem{prop}{\indent Proposition}

\newtheorem{lem}{\indent Lemma}
\newtheorem{defin}{\indent Definition}
\newtheorem{cor}{\indent Corollary}

\newcommand{\mmmintone}[1]{{\dis{\int\kern -.36cm-}}_{\kern-.21cm\substack{#1}}\;\;}
\newcommand{\mmmintwo}[2]{{\dis{\int\kern -.43cm-}}_{\kern-.21cm\substack{#1}}^{\substack{#2}}\;\;}
\newcommand{\submint}{{\scriptstyle{\int\kern -.66em -}}}
\newcommand{\submintone}[1]{{\scriptstyle{\int\kern -.66em-}}_{\scriptscriptstyle{\kern-.21em\substack{#1}}}}
\newcommand{\fracmint}{{\textstyle{\int\kern -.88em -}}}
\newcommand{\fracmintone}[1]{{\textstyle{\int\kern -.88em
-}}_{\scriptscriptstyle{\kern-.21em\substack{#1}}}\;}

\title{Layered systems at the mean field critical temperature}

\author{Luiz Renato Fontes\footnote{Instituto de Matem\'atica e
Estat\'\i stica. Universidade de S\~ao Paulo, SP, Brazil. E-mail:
lrfontes@usp.br},
Domingos H. U. Marchetti\footnote{Instituto de F\'\i sica. Universidade de
S\~ao Paulo, SP, Brazil. Email: marchett@if.usp.br},
Immacolata Merola\footnote{DISIM, Universit\`a di L'Aquila, L'Aquila, Italy.
Email: immacolata.merola@univaq.it},\\ Errico Presutti\footnote{GSSI,
L'Aquila, Italy. Email: errico.presutti@gmail.com}, and
Maria Eulalia Vares\footnote{Instituto de Matem\'atica. Universidade Federal
do Rio de Janeiro, RJ, Brazil. \!Email: eulalia@im.ufrj.br}}

\begin{document}

\maketitle

\begin{abstract}
We consider the Ising model on  $\mathbb Z\times \mathbb Z$  where
on each horizontal line $\{(x,i), x\in \mathbb Z\}$, the interaction
is given by  a ferromagnetic Kac potential with coupling strength $J_\ga(x,y)\sim\gamma J(\gamma (x-y))$
at the mean field critical temperature. We then add a nearest neighbor ferromagnetic
vertical interaction of strength $\epsilon$ and prove that for every $\epsilon >0$
the systems exhibits phase transition provided $\gamma>0$ is small enough.

\end{abstract}



\section{Introduction}
\label{sec;1}
We consider an Ising model on the lattice $\mathbb Z\times \mathbb Z$,
whose points we denote by $(x,i)$. The spins $\sigma(x,i)$ take values in $\{-1,+1\}$ and
on each horizontal line, also called \emph{layer}, $\{(x,i), x\in \mathbb Z\}$,
the interaction
is given by a ferromagnetic Kac potential, that is,
 the  interaction between the spins at $(x,i)$ and $(y,i)$ is given by
  \begin{equation}
  \label{1.1}
- \frac {1}2 J_\ga(x,y) \si(x,i)\si(y,i),\quad \sum_{y\ne x} J_\ga(x,y)=1,
   \end{equation}
where $J_\ga(x,y) = c_\ga \ga J(\ga (x-y))$; $J(r)$, $r\in \mathbb R$,
is a symmetric probability density with continuous derivative and support in $[-1,1]$, $\gamma >0$ is a scale parameter,
$c_\ga$ is the normalization constant ($c_\ga$ tends to 1 as $\ga\to 0$).  We also suppose that $J(0)>0$.
$H_{\ga,0}$ denotes the Hamiltonian with only the interactions \eqref{1.1} on each layer, so that
different layers do not interact with each other, the system is essentially
one dimensional and does not have phase  transitions.

We fix the inverse temperature at the mean field critical value
$\beta=1$ so that also in  the Lebowitz-Penrose limit no
phase transition is present.
Purpose of this paper is to study what happens if we
put a \emph{small nearest neighbor vertical interaction}
  \begin{equation}
  \label{1.3}
- \eps \;\si(x,i)\si(x,i+1).
   \end{equation}

The main result in this paper is the following.

\medskip
\begin{thm}
\label{thm1.1}
Given any $\eps>0$, for any $\ga>0$ small enough $\mu_{\ga}^{+}\ne \mu_{\ga}^{-}$,
 $\mu_{\ga}^{\pm}$
the plus-minus DLR measures
defined as the thermodynamic limits of the Gibbs measures
with plus, respectively minus, boundary conditions.

\end{thm}

\medskip

It is worth mentioning that a version of Theorem~\ref{thm1.1} holds for $\beta>1$ with $\eps=\gamma^A$ for any
$A$. (See~\cite{FPPMV1} where indeed the above result has been conjectured.)

In many cases it has been proved that
if in the Lebowitz-Penrose limit
there is a phase transition then in dimension $d\ge 2$ there is also a phase transition
at small $\ga>0$ (i.e.\ without taking the limit $\ga\to 0$). We cannot follow this route here
because
we do not know the phase diagram for our model in the limit $\ga \to 0$: a
``Lebowitz-Penrose theorem" for our system is an interesting open problem that our analysis
does not solve. If the support of the Kac interaction would contain two dimensional balls
(i.e.\ layers at distance of order $\ga^{-1}$ interact with each other)
then the Lebowitz-Penrose analyis  \cite{LP} would apply and therefore
the free energy in the limit $\ga\to 0$ would be the convex envelope (i.e.\ the Legendre transform of the Legendre transform)
        \begin{equation}
        \label{1t.1}
\Big(f_\eps (m) - \frac{m^2}{2}\Big)^{**},
    \end{equation}
where $f_\eps (m)$ is the free energy of the one dimensional Ising model
with nearest neighbor interaction of strength $\eps$. \eqref{1t.1} yields a phase transition
if $\eps>0$.  Does  \eqref{1t.1} remain valid also when the Kac interaction is only horizontal?
We do not know the answer but our analysis shows that indeed our system has a phase transition
as indicated by  \eqref{1t.1}.

%


The proof of  Theorem \ref{thm1.1} requires a non trivial extension of previous
results on Kac potentials and it is given in complete details in this paper.
It is obtained by proving
Peierls bounds for suitably defined contours.
The bounds are established via a Lebowitz-Penrose
coarse graining procedure which however is not
straightforward for the reasons explained before (due to the local nature of the
vertical interaction  and the strictly horizontal structure
of the Kac interaction). The trick is to use ferromagnetic inequalities to compare
the magnetization under $\mu^+_\gamma$ with that under the corresponding
Gibbs measure for which the vertical interaction is removed in a chessboard fashion. To this new system (which is
more decoupled but not so much as to lose the phase transition) we can apply the  Lebowitz-Penrose
coarse graining strategy. In this way we
reduce the proof of the Peierls bound to the analysis
of  variational problems for a suitable free energy functional.

The model we are considering is related to a $d=1$
quantum spin model with transverse field, whose hamiltonian is:
\[
H(\si) = - \sum_{x \ne y} J_\ga(x,y) \hat \si^3(x) \hat \si^3(y)
- \alpha \sum_x  \hat \si^1(x)
\]
in its stochastic representation via  Feynman-Kac, \cite{AKN},
\cite{CCIL}  and \cite{IL}.  We are indebted to D. Ioffe for pointing out
the connection and for useful comments.

At this point we state two conjectures.

The first question is: can we choose $\eps= \eps(\ga)$
so that $\eps(\ga)\to 0$  as $\ga\to 0$ and
still have for all $\ga$ small enough a phase transition?  Is there a critical
choice for $\eps(\ga)$?  The conjecture is that
setting $\eps(\ga)= \kappa \ga^{2/3}$, $\kappa>0$,  we have a phase transition if $\kappa$
is large enough and no phase transition for $\kappa$ small.

This is related
to the next conjecture.  Consider the system where on each layer we have a   process $m(r,i)\in \mathbb R$,
$r\in \mathbb R$, $i\in \mathbb Z$.  The formal Gibbs measure that we want to study is:
           \begin{equation}
        \label{1t.4}
 e^{- \sum_i\{\int dr \kappa m(r,i)m(r,i+1)\}} \prod_i P(dm(r,i))
    \end{equation}
where $P$ is the Euclidean $\phi^4_1$ probability measure, namely the stationary solution of the real valued stochastic PDE
           \begin{equation}
        \label{1t.5}
du(r,t) = \Big (\frac 12 u''(r,t) - u^3(r,t)\Big)dt + dw,
    \end{equation}
$dw$ white noise in space-time.
The conjecture is that there is a  phase transition for $\kappa$ large and no phase transition for $\kappa$ small.

The measure in \eqref{1t.4} is the formal scaling limit of the Gibbs distribution
of the empirical magnetization when we scale space as $x \to r= \ga^{1+1/3}x$
and renormalize the averages by a factor $\ga^{-1/3}$ as proved
in \cite{BPRS} and \cite{FR}; see also \cite{D}, where \eqref{1t.5}
without the second derivative term  is derived
by studying the critical fluctuations in the mean field version of the model.

More precisely, in both papers the question is about the analysis of the long space-time
fluctuations of the d=1 Ising model with Glauber dynamics
and Kac potential at $\beta=1$ (like ours here).  Namely the analysis of the fluctuations field
\[
\ga\sum_x \phi(\ga^{1+1/3} x) \si(x, \ga^{-2/3}t),
\]
 with $\phi$ a  test function.  This
is the right normalization because
one can prove that at such long times
the typical values of the empirical magnetization in the limit $\ga\to 0$
have order $\ga^{1/3}$ and not the normal values $\ga^{1/2}$ of the finite time fluctuations.
It is then proved that the above fluctuations field converges to $\int \phi(r)u(r,t) dr$
where $u$ solves \eqref{1t.5}.

Outline of the paper: As already mentioned, our proof involves the study of the Gibbs measures for another
Hamiltonian, denoted by $H_{\gamma,\epsilon}$ and defined in \eqref {a2.3}, where the vertical
interactions are removed in a convenient chessboard fashion. This yields a
two dimensional system with pairs of long segments of parallel layers interacting vertically within the pair
(but not with the outside), plus the horizontal Kac interaction. For this system we can exploit the spontaneous
magnetization that emerges when two parallel one dimensional Kac models at mean field critical temperature interact
vertically as in our case, as studied in Section \ref{sec:4s}. This spontaneous magnetization plays a natural role
in the definition of contours (as in the analysis of the one dimensional Kac interactions below the mean field
critical temperature). The main point is that for this chessboard Hamiltonian, and after a proper coarse graining
procedure, we are able to implement the Lebowitz-Penrose procedure: the corresponding free energy functional is
defined in Section \ref{sec:2s} and the problem of getting the corresponding Peierls bounds for the weight of
contours is transformed in variational problems for the free energy functional. This is the content of Theorem
\ref{thm3s.3}, whose proof involves, as a preliminary step, the study of the free energy function of two
layers and its minimizers (determining the spontaneous magnetization). There are delicate choices of scales so as
to allow the implementation of this procedure, as explained in Section \ref{sec:2}. In Section \ref{sec:4t} we get
an upper bound for the restricted partition function that appears in the numerator of the weight of a contour.
Section \ref{sec:4s} is dedicated to the study of the minimizers of the free energy functional of
Section \ref{sec:2s}, which then will play a crucial role in the lower bound for denominator
in the weight of a contour, completed in Section \ref{sec:6t}. In Section \ref{sec:7t} we combine the two
estimates to conclude the proof of the main theorem. The analysis of the mean field free energy function for two
layers and the crucial estimates used in Section \ref{sec:4s} are carried out in the Appendices.

\vskip1cm

\setcounter{equation}{0}

\section{Contours}
\label{sec:2}

For $i\in \mathbb Z$, we call $i$-th layer the set
$\mathbb Z \times \{i\}$. As mentioned in the introduction we shall extensively use
coarse graining, for which we start by partitioning  each layer
into intervals of length  $\ell \in \{ 2^n,\;n\in \mathbb Z\}$.
Let $\mathcal D^{\ell,i}=\{C^{\ell,i}_{k \ell}, k\in \mathbb Z\}$
denote the partition of the $i$-th layer:
  \begin{equation}
        \label{2.1}
C^{\ell,i}_{x}=C^{\ell}_{x} \times \{i\}:= ([k\ell, (k+1)\ell) \cap \mathbb Z)\times \{i\}, \; \text{where}\; k=\lfloor x/\ell\rfloor
\end{equation}
and, as usual, $\lfloor s\rfloor=\max\{x \in \mathbb Z\colon x \le s\}$.
To simplify notation we  restrict
$\ga$ to belong to $\{2^{-n}, n\in \mathbb N\}$.

For the coarse grained description we shall use  three length scales and an accuracy parameter
$\zeta>0$ which all depend on $\ga$:
    \begin{equation}
    \label{2.2}
\ga^{-1/2},\;\; \ell_{\pm}= \ga^{-(1 \pm \alpha)},\quad \zeta=
\ga^a,\qquad 1\gg \alpha\gg a>0.
     \end{equation}
The smallest scale, $\ga^{-1/2}$, will be used to implement the Lebowitz-Penrose
procedure to define free energy functionals.  Together with $\zeta$, the scales $\ell_{-}$ and $\ell_+$ will be used to define,
at the spin level, the \emph{plus} and \emph{minus} regions and then the contours.

For notational simplicity we suppose that also $\ga^{-\alpha}$ 
and all the above lengths belong to
$\{2^{n}, n\in \mathbb N_+\}$: this is a restriction on
$\alpha$ that could be removed by changing ``slightly'' $\alpha$ with $\ga$.

We shall prove Theorem \ref{thm1.1} for the ``chessboard" Hamiltonian:
  \begin{equation}
  \label{a2.3}
H_{\ga,\eps} = -\frac {1}2 \sum_{x\ne y,i} J_\ga(x,y) \si(x,i)\si(y,i) - \eps\sum_{x,i} \chi_{i,x}
\si(x,i)\si(x,i+1),
   \end{equation}
where
  \begin{equation}
  \label{a2.4}
 \chi_{x,i} = \begin{cases} 1 & \text{if $\lfloor x/{\ell_+}\rfloor+i$ is even}, \\
 0 & \text{otherwise}.
 \end{cases}
   \end{equation}

\begin{defin}
\label{a2.4.1}
 When $\chi_{x,i}=1$, according to \eqref{a2.4}, we say that $(x,i)$ and $(x,i+1)$ interact vertically and
denote by $v_{x,i}$ the site $(x,j)$ which
 interacts vertically with $(x,i)$.
 \end{defin}

 By the GKS correlation inequalities (see e.g.~Theorem 1.21 in Chapter IV of~\cite{L}),
the plus state for $H_{\ga,\eps}$ is less magnetized
than the one for the full Hamiltonian (with $\chi_{i,x}$  replaced by 1 everywhere).
Hence Theorem \ref{thm1.1} will follow once we prove that the
magnetization in the  plus state
of the  Hamiltonian  given by \eqref{a2.3} is strictly positive.

For the chessboard Hamiltonian, we shall see via a Lebowitz-Penrose analysis
that in the limit as $\ga\to 0$ there is a spontaneous magnetization equal to
some  $m_\eps>0$ in the plus state and $-m_\eps$ in the minus state. This will follow from the
analysis in sections \ref{sec:2s}--\ref{sec:4s}. This value $m_\eps$ is used to define
contours, as we now explain (taking $m_\eps>0$ for granted).

Define first the empirical magnetization on a scale $\ell \in \{2^n, n\in \mathbb N\}$ in the layer $i$ as
    \begin{equation}
    \label{2.3}
\si^{(\ell)}(x,i)
:= \frac{1}{\ell} \sum_{y: (y,i)\in C^{\ell,i}_{x}} \si(y,i).
     \end{equation}
We also consider the partition of $\mathbb Z^2$ into rectangles $\{Q_\gamma (k,j)\colon k, j \in \mathbb Z\}$, where
$$
Q_\gamma (k,j)=\left([k\ell_+, (k+1)\ell_+) \times [j\ga^{-\alpha}, (j+1)\ga^{-\alpha})\right) \cap \mathbb Z^2 \text{  if $k$ is even}
$$
and
$$
Q_\gamma (k,j)=\left([k\ell_+, (k+1)\ell_+) \times (j\ga^{-\alpha}, (j+1)\ga^{-\alpha}]\right) \cap \mathbb Z^2\text{  if $k$ is odd}.
$$
For convenience we sometimes write $Q_{x,i}=Q_\gamma (k,j)$ if $(x,i) \in Q_\gamma (k,j)$.
The important feature of this definition (frequently exploited in the sequel) is that the spins in each rectangle
$ Q_{x,i}$ do not interact vertically with the spins of the complement, namely
recalling the
definition of $v_{x,i}$ and that $\ga^{-\alpha}$ is even, we see that
 $v_{x,i} \in Q_{x,i}$ for all $(x,i)$.  Notice also that the rectangles
$Q_\gamma (k,j)$ become squares if lengths are measured in interaction length units: in fact
in such units the horizontal side of a rectangle has length $\ell_+ / \ga^{-1}= \ga^{-\alpha}$
and  the vertical side has also length $\ga^{-\alpha}$ as the vertical
interaction length is equal to 1. The other important feature
behind the definition of rectangles is that their size in interaction length units diverges
as $\ga\to 0$: this will be exploited to prove decay of correlations from the boundaries.

 \medskip

The random variables $\eta(x,i)$, $\theta(x,i)$ and $\Theta(x,i)$ are then defined as follows:

\begin{itemize}

\item\;   $\eta(x,i)=\pm 1$ if  $\dis{ \big|  \si^{(\ell_-)}(x,i) \mp m_\eps
 \big|\le \zeta}$; $\eta(x,i)=0$ otherwise.

\item\;   $\theta(x,i)= 1$, [$=-1$], if
$\eta(y,j)= 1$, [$=-1]$, for all $(y,j)\in
Q_{x,i}$;  $\theta(x,i)=0$ otherwise.

\item\;   $\Theta(x,i)= 1$, [$=-1$], if
$\eta(y,j)= 1$, [$=-1$], for all $(y,j) \in
\cup_{u,v\in \{-1,0,1\}} Q_\gamma(k+u,j+v)$,
with $(k,j)$ determined by $Q_{x,i} =Q_\gamma(k,j)$.

\end{itemize}
Namely, for the $\Theta$ variables we consider a ``block"  $3\times 3$ of $Q$-rectangles.

\medskip

\noindent
The plus phase is the union of all
the rectangles
$Q_{x,i}$ such that  $\Theta(x,i)=1$,
the  minus phase is where   $\Theta(x,i)=-1$, in the complement the phase
is \emph{undetermined}.

Two rectangles $Q_\gamma (k,j)$ and $Q_\gamma (k',j')$ are said to
be connected if $(k,j)$ and $(k',j')$ are $*$--connected, i.e. $|k-k'|\vee |j-j'|\le 1$.
By choosing suitable boundary conditions, we shall restrict in the sequel
to spin configurations
such that $\Theta=1$ outside of a compact
(the case
when $\Theta=-1$ 
can be recovered
via spin flip). Given such a  $\si$,
%
we call \emph{contours} the  pairs
$\Ga=({\rm sp}(\Ga), \eta_\Ga)$, where ${\rm sp}(\Ga)$ is a maximal connected component
of the undetermined region, called \emph{the spatial support of $\Ga$},
and $\eta_\Ga$ is the restriction of $\eta$ to ${\rm sp}(\Ga)$,  called \emph{the specification of $\Ga$}.

\vskip.5cm
{\bf Geometry of contours.}
Denote by ${\rm ext} (\Ga)$ the maximal unbounded
connected component of the complement of ${\rm sp}(\Ga)$
and $\partial_{\rm out}(\Ga)$ the union
of the rectangles in ${\rm ext} (\Ga)$ which are
connected to ${\rm sp}(\Ga)$. $\partial_{\rm in}(\Ga)$
is instead the union of the rectangles in
${\rm sp} (\Ga)$ which are connected to ${\rm ext} (\Ga)$.
$\Theta$ is constant and different from 0 on $\partial_{\rm out}(\Ga)$
and we call \emph{plus} a contour $\Ga$
when $\Theta=1$   on $\partial_{\rm out}(\Ga)$ and \emph{minus} otherwise.
Observe that in a plus contour $\eta=1$ on $\partial_{\rm in}(\Ga)$.

Analogously we call ${\rm int}_k(\Ga), k=1,\dots,k_\Ga$ the bounded
maximal connected components (if any)
of the complement of ${\rm sp}(\Ga)$,
$\partial_{{\rm in},k}(\Ga)$ the union of all rectangles in
${\rm sp}(\Ga)$ which are connected
to ${\rm int}_k(\Ga)$. $\partial_{{\rm out},k}(\Ga)$ is
the union of all  the rectangles in
${\rm int}_k(\Ga)$ which are connected
to ${\rm sp}(\Ga)$. Then
$\Theta$ is constant and different from 0
on each $\partial_{{\rm out},k}(\Ga)$ and
we write $\partial^{\pm}_{{\rm out},k}(\Ga)$, ${\rm int}^\pm_k(\Ga)$  and $\partial^{\pm}_{{\rm in},k}(\Ga)$
if $\Theta=\pm 1$ on the former, observing that $\eta = \pm 1$ on
$\partial^{\pm}_{{\rm in},k}(\Ga)$, respectively.  We also call
    \begin{equation}
    \label{2.4}
c(\Ga) = {\rm sp}(\Ga) \cup \bigcup_{k} {\rm int}_k(\Ga).
     \end{equation}

\vskip.5cm
{\bf Diluted partition functions.}
Let  $\La$ be a bounded region which is an union of $Q$-rectangles. 
The plus diluted partition function in $\La$ with boundary conditions  $\bar\si$ is
    \begin{equation}
    \label{2.4.2}
Z^+_{\La,\bar\si} = \sum_{\si_{\La}} \mathbf 1_{\{\Theta =1\, \text{on $\partial_{\rm in} (\La)$}\}}
e^{- H_{\ga,\eps}(\si_{\La}|\bar\si)} =: Z_{\La,\bar\si}(\Theta =1\, \text{on $\partial_{\rm in} (\La)$}),
     \end{equation}
where $\bar\si$ is a configuration on the complement of $\La$;
$\Theta$ is computed on the configuration $(\si_{\La},\bar\si)$ and
$\partial_{\rm in} (\La)$ is the union of all $Q$-rectangles in $\La$ connected to
$\La^c$.  Minus diluted partition functions are defined analogously. As a rule we denote by
$Z_{\La,\bar\si}(\mathcal A)$ the partition function with the constraint
$\mathcal A$, $\mathcal A$ a set of configurations. Notice that there is no  vertical
interaction between the spins in $\La$ and those in its complement because $\La$ is union of rectangles.

 The plus diluted Gibbs measure (with boundary conditions $\bar\si$) is defined in the
usual way, namely, given a configuration of spins $\si_\La$ on $\La$, the weight assigned
to $\si_\La$ by the plus Gibbs measure is given by
\begin{equation}
    \label{eq:pmgm}
\mu^+_{\La,\bar\si}(\si_\La) = \frac{e^{- H_{\ga,\eps}(\si_{\La}|\bar\si)}}{Z^+_{\La,\bar\si}}
\mathbf 1_{\{\Theta =1\, \text{on $\partial_{\rm in} (\La)$}\}}.
     \end{equation}
The minus diluted Gibbs measure is defined analogously.

 We shall prove the
Peierls estimates for the plus and minus diluted Gibbs measures, which, as a consequence,
 have distinct thermodynamic limits;
Theorem \ref{thm1.1} will then follow.

\vskip.5cm
{\bf Weight of a contour.}
We are now ready to define the fundamental notion of \emph{weight of a contour}.  Let
$\Ga$ be a plus contour (the definition for minus contours is obtained by spin flip)
and $\bar \si$ a configuration on the complement of $c(\Ga)$ such that $\eta=1$ on
$\partial_{\rm out}(\Ga)$ (in agreement with the definition of a plus contour).
Then the weight of $\Ga$ with boundary conditions
$\bar \si$ is
    \begin{equation}
    \label{2.6}
W_\Ga(\bar\si):= \frac
{Z_{c(\Ga);\bar\si}(\eta  = \eta_\Ga \;{\rm on}\;
 \, {\rm sp}(\Ga); \Theta = \pm 1 \;{\rm on \;each} \,\; \partial_{{\rm out},k}^{\pm}(\Ga))}
{Z_{c(\Ga);\bar\si}(\Theta = 1\; \rm{ on}\;
{\rm sp}(\Ga)\; {and \; on\; each}\, \partial_{{\rm out},k}^{\pm}(\Ga)\})}
     \end{equation}
where $Z_{\La,\bar \si}(\mathcal A)$ is the partition function
in $\La$ with Hamiltonian $H_{\ga,\eps}$,
with boundary conditions $\bar\si$ and constraint $\mathcal A$.  In the next sections we shall prove the following theorem

\begin{thm} [The Peierls bounds]
\label{peierls}

There are $c>0$, $\eps_0>0$ and $\ga_\cdot:(0,\infty)\to(0,\infty)$ so that for any $0<\eps\le \eps_0$,
$0<\ga\le \ga_\eps$ and any contour $\Ga$ with boundary spins $\bar\si$
  \begin{equation}
    \label{2.66}
    W_\Ga(\bar\si) \le e^{-c |{\rm sp}(\Ga)| \ga^{2a+4\alpha}}.
    \end{equation}

\end{thm}

In Section \ref{sec:7t} we shall see how to prove Theorem \ref{thm1.1} using the Peierls bounds
\eqref{2.66}.

\vskip1cm

\setcounter{equation}{0}

\section{Reduction to a variational problem} 
\label{sec:2s}

The goal of this section is to introduce the Lebowitz-Penrose free energy functional
and to set the variational problem that emerges in the estimates
of the partition functions in \eqref{2.6}.
We start by the next proposition which deals with the very simple situation of two layers
of $\pm 1$ spins whose unique interaction is the nearest neighbor vertical one. It is
just a chain of independent pairs of spins. Therefore the multi-canonical partition
function, where we fix the magnetization on each layer, is studied by very simple tools.
This first  result, proved in Appendix \ref{sec:5s}
for sake of completeness,
describes its convergence (in the thermodynamic limit) to  the infinite volume free energy
$\hat\phi_\eps(m_1,m_2)$ and finite volume corrections. We then state and prove a Lebowitz-Penrose
theorem for the spin model associated to the chessboard Hamiltonian $H_{\gamma,\eps}$.

\medskip

\begin{prop}
\label{thm3s.2}

Let $n$ be a positive integer,  $X_n=\{-1,1\}^n$.  For $i=1,2$, let \\
$m_i \in \{-1 + \frac{2j}{n} \colon j=1,\dots,n-1\}$  and set
    \begin{equation}
    \label{2ss.8}
Z_{\eps,n}(m_1,m_2) = \sum_{(\si_1,\si_2)\in X_n \times X_n} \mathbf 1_{\{ \sum_{x=1}^n \si_i(x) =n m_i\; i=1,2\}}
e^{ \eps \sum_{x=1}^n\si_1(x)\si_2(x)}.
    \end{equation}
    There is a continuous and convex function $\hat \phi_\eps$ defined on $[-1,1]\times [-1,1]$,  with
    bounded derivatives on each $[-r,r]\times [-r,r]$ for $|r| <1$, and a constant $c>0$ so that
    \begin{equation}
    \label{2ss.9}
- \hat \phi_\eps (m_1,m_2) - c\frac{\log n}{n} \le     \frac 1n \log Z_{\eps,n}(m_1,m_2) \le -
\hat \phi_\eps (m_1,m_2). 
    \end{equation}

\end{prop}

\medskip

We shall next use the above proposition to study
the partition functions which enter in the definition of contours.
We thus consider a region  $\La$ which in the applications will
be the spatial support of a contour. Here it only matters that
$\La$ is a connected set union of $Q$-rectangles. We want to bound from above and below
the partition function
    \begin{equation}
    \label{2s.12}
Z_{\La,\bar \si}(\mathcal A) := \sum_{\si_{\La} \in \mathcal A} e^{-H_{\ga,\eps}(\si_{\La}\,|\,\bar\si)},
    \end{equation}
where $\bar \si$ is a spin configuration in the complement of $\La$ and ``the constraint''
$\mathcal A$
is a set of configurations in $\La$ defined in terms of the values of $\eta_\La$. 
We shall coarse-grain on the scale $\ga^{-1/2}$.  We thus call $M_{\ga^{-1/2}}$
the possible values of the magnetization densities
$\si^{(\ga^{-1/2})}$, ($\si^{(\ell)}$ has been defined in \eqref{2.3}),
namely
\[
M_{\ga^{-1/2}} = \{-1, -1+2 \ga^{1/2},...,1-2 \ga^{1/2}, 1\}
\]
and we set
    \begin{equation}
    \label{2s.13}
\mathcal M_{\La}:= \{m(\cdot)\in (M_{\ga^{-1/2}})^{\La}: \; \text{$m(\cdot)$ is constant on
each  $C^{\ga^{-1/2},i}\subset \La$}\}.
    \end{equation}
The Lebowitz-Penrose free energy functional (on $\La$ with boundary conditions $\bar m$) is the following functional on
$[-1,1]^{\La}$ (whose elements are denoted in short by $m$)
        \begin{eqnarray}
    \label{2s.14}
F_{\La,\ga}( m | \bar m) &=& \frac 12 \sum_{(x,i)\in \La}
\hat \phi_{\eps} (m(x,i),m(v_{x,i})) \nn\\
&-& \frac 12\sum_{(x,i)\ne(y,i) \in \La} J_\ga(x,y) m(x,i)m(y,i) \nn\\
&-&
\sum_{(x,i)\in \La,\; (y,i)\notin \La} J_\ga(x,y) m(x,i)\bar m(y,i),
    \end{eqnarray}
where $\bar m \in  [-1,1]^{\La^c}$, $\hat \phi_{\eps}$ is the free energy function in \eqref{2ss.9} and
$ v_{x,i}$ is given in Definition \ref{a2.4.1}. (Recall that $v_{x,i} \in \La$ for each $(x,i) \in \La$
since there are no vertical interactions between a $Q$--rectangle and the outside.)

\emph{Notational remark.} The same formula is used when $\bar m$ is defined in a set $\Delta$
contained in the complement of $\La$; in such a case the sum over
$(y,i)$ in the last term is extended only to $\Delta$.

By an abuse of notation we write, analogously to \eqref{2.3},
   \begin{equation}
    \label{2s.14.1}
m^{(\ell)}(x,i)
:= \frac{1}{\ell} \sum_{y: (y,i)\in C^{\ell,i}_{x}} m(y,i)
     \end{equation}
and define $\eta(x,i;m)=\pm 1$ if  $\dis{ \big|  m^{(\ell_-)}(x,i) \mp m_\eps
 \big|\le \zeta}$
 and $=0$ otherwise.
We
still denote by $ \mathcal A$ a constraint that depends on $\eta(\cdot;m)$ as for instance $\eta(\cdot;m)=\eta^*(\cdot)$
on $\La$.

\vskip.5cm

\begin{thm}
\label{thm3s.3}
There is a constant $c$ so that
    \begin{equation}
    \label{2s.15.1}
  \log  Z_{\La}(\bar \si; \mathcal A) \le  - \inf_{m\in \mathcal M_{\La}\cap \mathcal A}F_{\La,\ga}( m | \bar m)
    + c |\La| \ga^{1/2}\log \ga^{-1},
    \end{equation}
where, recalling  \eqref{2.3}, $\bar m (x,i) ={\bar\si}^{ \ga^{-1/2}}(x,i) $, $(x,i)\notin \La$.
Moreover, for any $m\in  \mathcal M_{\La}\cap\mathcal A$
    \begin{equation}
    \label{2s.16}
 \log   Z_{\La}(\bar \si; \mathcal A) \ge  - F_{\La,\ga}( m | \bar m)
 -c |\La| \ga^{1/2}\log \ga^{-1}.
     \end{equation}
\end{thm}

\begin{proof}
 Here is essential the
restriction to regions with no
vertical interaction with the complement.
We have, writing $\si$ for a spin configuration in $\La$,
    \begin{equation*}
    \label{6s.1}
Z_{\La}(\bar \si; \mathcal A) = \sum_{m\in \mathcal A}\;\sum_{\si: \si^{(\ga^{-1/2})}(\cdot)
= m(\cdot) \text{ on } \La} e^{-H_{\ga,\eps;\La}(\si|\bar\si)},
    \end{equation*}
where $H_{\ga,\eps;\La}(\si|\bar\si)$ is the Hamiltonian \eqref{a2.3}
in the region $\La$ interacting with $\bar\si$ outside $\La$.

By the smoothness of $J_\ga$ we get
    \begin{equation*}
| \frac {1}2 \sum_{(x,i) \in \La}\sum_{y\ne x} \Big(J_\ga(x,y)-\hat J_\ga(x,y)\Big) \si(x,i)\si(y,i) | \le c |\La| \ga^{1/2},
    \end{equation*}
where $\si(y,i)=\bar\si(y,i)$ if $(y,i)\notin \La$ and
    \begin{equation*}
\hat J_\ga(x,y) :=  \frac{1}{\ga^{-1}} \sum_{x' \in C^{\ga^{-1/2}}_x}
\sum_{y' \in C^{\ga^{-1/2}}_y} J_\ga(x',y').
    \end{equation*}
Thus, recalling that there is no vertical interaction between $\La$ and its complement,
  \begin{equation*}
| H_{\ga,\eps;\La}(\si|\bar\si)-\big(H_{\ga,0;\La}(m|\bar m)   -\frac 12 \eps\sum_{(x,i)\in \La}
\si(x,i)\si(v_{x,i})\big)| \le c |\La| \ga^{1/2},
   \end{equation*}
where $H_{\ga,0;\La}(m|\bar m)$ is the Hamiltonian $H_{\ga,\eps;\La}(\si|\bar \si)$
with $\eps=0$ and the spins replaced by $m(x,i)$ and $\bar m(x,i)$.
We then get, using \eqref{2ss.9},
        \begin{equation*}
 Z_{\La}(\bar \si; \mathcal A) \le | \mathcal M_{\La}|\sup_{m\in \mathcal A}
e^{- H_{\ga,0;\La}(m|\bar m)- \frac 12 \sum_{(x,i)\in \La}
\hat\phi_{\eps} (m(x,i),m(v_{x,i}))} e^{c |\La| \ga^{1/2}},
     \end{equation*}
which proves \eqref{2s.15.1} because $| \mathcal M_{\La}| \le (c\ga^{-1/2})^{|\La| \ga^{1/2} }$ (for a suitable constant $c$).
\eqref{2s.16} is proved similarly.
\end{proof}

\medskip

 The variations of
$J_\ga$ on the scale $\ga^{-1/2}$ give a contribution of the order  $ |\La| \ga^{1/2}$ to the errors in  \eqref{2s.15.1}
and \eqref{2s.16}; in \eqref{2s.15.1} there is also a contribution of order $|\La| \ga^{1/2} \log \ga^{-1}$ coming from
the cardinality of $\mathcal M_{\La}$. In \eqref{2s.16} we need to take into account the lower bound in Proposition \ref{thm3s.2}.
Of course in the upper bound of the partition function we can drop the condition that $m$ takes values in
$M_{\ga^{-1/2}}$ and that it is constant in the intervals $C^{\ga^{-1/2},i} \subset \La$:

\begin{cor}
\label{cor3s.1}
In the same context of  Theorem \ref{thm3s.3}
    \begin{equation}
    \label{2s.15}
  \log  Z_{\La}(\bar \si; \mathcal A) \le  - \inf_{m\in [-1,1]^{\La}\cap \mathcal A}F_{\La,\ga}( m | \bar m)
    +c|\La| \ga^{1/2} \log \ga^{-1}.
    \end{equation}

\end{cor}

\vskip1cm

\setcounter{equation}{0}

\section{The upper bound}
\label{sec:4t}

Let us now be more specific and see how \eqref{2s.15}  is used to get an upper bound for
the numerator of \eqref{2.6}. The key point will be to prove that the excess free energy
due to the constraint $\eta=\eta_\Ga$ is much larger than
the errors in \eqref{2s.15.1}--\eqref{2s.16}.

In the sequel we specify $\La={\rm sp}(\Ga)$,
and refer to the  paragraph ``Geometry of contours'' in Section
\ref{sec:2}. Because the notation gets clumsy in some formulae,
we shorten it a bit as follows:
\[\Delta_{\rm in}= \partial_{\rm in}(\Ga),\;
\Delta^{\pm}_k=\partial^{\pm}_{{\rm in},k}(\Ga),\;
I^\pm_k={\rm int}^{\pm}_k (\Ga),
\]
recalling that
the suffix $\pm$ here  refers to the (constant) value of $\Theta$ on the corresponding $\partial_{{\rm out},k}(\Gamma)$,
and that $\eta=\pm 1$ on $\Delta^{\pm}_k$. Set then
\[
\Delta_0={\rm sp}(\Ga) \setminus (\Delta_{\rm in}  \cup \{\cup_k\Delta_{k}^+\}\cup\{ \cup_k\Delta_{k}^-\})
\]
so that one has the following partition of $c(\Ga)$:
\begin{equation}\label{eq:c}
 c(\Ga)= \Delta_0 \cup\Delta_{\rm in}  \cup \{\cup_k\Delta_{k}^+\}
 \cup \{\cup_k\Delta_{k}^-\}\cup \{\cup_k I^+_k\}\cup \{\cup_k I^-_k\}.
\end{equation}


Thus the function $\bar m$ in  \eqref{2s.15} is specified by the spins outside
$c(\Ga)$ and by those in the sets
$I_k^{\pm}$.  When necessary we
write $\bar m_{\si_{\rm ext}}$ $\bar m_{\si_{I_k^{\pm}}}$ for its restriction to the
complement of $c(\Ga)$ and to $I_k^{\pm}$, respectively.
Finally the constraint in Theorem \ref{thm3s.3} is
$\mathcal A
= \{ \eta =\eta_\Ga \;{\rm on}\;\La \}$.

Following \cite{presutti} (see Chapters 6 and 9), an important ingredient in the proof of Peierls bounds consists in showing
that the minimizers of the free energy functional have good regularity properties even when constrained to have given
magnetization values in \emph{small boxes} (the multi-canonical constraints).
The next proposition, proved in Appendix \ref{sec:7s},
shows that the infimum in \eqref{2s.15} can be restricted to \emph{smooth functions}.

\vskip.5cm

\begin{prop}
\label{thm3s.4} 
There is a positive constant $c$ so that, with the same notation as above
and recalling that $\La={\rm sp}(\Ga)$,
     \begin{equation}
    \label{2s.16.1}
 \inf_{m\in [-1,1]^{\La}\cap\mathcal A}F_{\La,\ga}( m | \bar m) =
 \inf_{m\in  [-1,1]^{\La}\cap \mathcal A\cap S_{\Delta_0}}F_{\La,\ga}( m | \bar m),
    \end{equation}
where
     \begin{equation}
    \label{2s.16.1.1}
    S_{\Delta_0} := \{ m:
\sup_{(x,i)\in \Delta_0} |m(x,i) - m^{(\ell_-)}(x,i)| \le c \ga^\alpha \}.
    \end{equation}


\end{prop}

\smallskip

\emph{Remark.}
The smoothness request could be extended to the whole
$\La$ without changing the infimum but we only need it in $\Delta_0$.

\medskip
Let us write
   \begin{equation}
    \label{2s.16.1.1.1}
-m(x,i) m(y,i) = \frac 12 \Big( -m(x,i)^2- m(y,i)^2 + [m(x,i) -m(y,i)]^2\Big)
  \end{equation}
in some of the terms.

With the above notation, and recalling~(\ref{2s.14}), we get
    \begin{eqnarray}
    \label{4s.10}
F_{{\rm sp}(\Ga),\ga}( m | \bar m) &=& F^*_{\Delta_0,\ga}( m_{\Delta_0}) +
  F'_{\Delta_{\rm in},\ga}( m_{\Delta_{\rm in}}|\bar m_{\si_{\rm ext}})\nn\\ &+&
 \sum_{k} F'_{\Delta^+_k,\ga}( m_{\Delta^+_k}|\bar m_{\si_{I_k^{+}}})
  +
 \sum_{k} F'_{\Delta^-_k,\ga}( m_{\Delta^+_k}|\bar m_{\si_{I_k^{-}}})\nn\\ &+& \frac 12
 \sum_{(x,i)\in \Delta_0}\; \sum_{(y,i)\notin \Delta_0} J_\ga(x,y)
 [m(x,i) -m(y,i)]^2,
     \end{eqnarray}
where, writing $m$ for $ m_{\Delta_0}$,
       \begin{eqnarray}
    \label{2s.19}
F^*_{\Delta_0,\ga}( m) &=& \sum_{(x,i)\in \Delta_0} \{ -\frac 12 m(x,i)^2+\frac 12
\hat\phi_{\eps} (m(x,i),m(v_{x,i})) \}\nn\\
&+& \frac 14\sum_{(x,i)\ne(y,i) \in \Delta_0} J_\ga(x,y) (m(x,i)-m(y,i))^2,
    \end{eqnarray}
while writing $m$ for $m_{\Delta_{\rm in}}$,
       \begin{eqnarray}
    \label{2s.20}
F'_{\Delta_{\rm in},\ga}( m|\bar m_{\si_{\rm ext}})  &=& F_{\Delta_{\rm in},\ga}( m|\bar m_{\si_{\rm ext}})
- \sum_{(x,i)\in \Delta_{\rm in}}a_{x,i}\frac{m(x,i)^2}{2},
    \end{eqnarray}
where
   \begin{equation}
    \label{4s.5.23}
a_{x,i} := \sum_{y:(y,i)\in \Delta_0} J_\ga(x,y).
    \end{equation}
$ F'_{\Delta^{\pm}_{k},\ga}(  m_{\Delta^{\pm}_k}|\bar m_{\si_{I_k^{\pm}}})$ is
defined analogously.


\emph{Remark.}
Recalling the remark after \eqref{2s.14}, notice that the interaction
between $\Delta_{\rm in}$ and $\Delta_{0}$ present in
$F_{{\rm sp}(\Ga),\ga}( m | \bar m)$
is absent in $F_{\Delta_{\rm in},\ga}( m_{\Delta_{\rm in}}|\bar m_{\si_{\rm ext}}$.  It is instead contained
(and split via \eqref{2s.16.1.1.1})  in the following terms:
(i) the last term on the right hand side of \eqref{4s.10},
(ii) the second term in the right hand side of \eqref{2s.20} 
with
  $m(x,i)^2$, $(x,i)\in \Delta_0$.
Something analogous takes place to the interaction between $\Delta^\pm_k$ and $\Delta_0$, involving the last term on the rhs of (4.4) and the corresponding terms in $ F'_{\Delta^{\pm}_{k},\ga}(  m_{\Delta^{\pm}_k}|\bar m_{\si_{I_k^{\pm}}})$.


By Proposition \ref{thm3s.4} and~(\ref{4s.10}), 
we get,
dropping the last term in \eqref{4s.10},

\begin{cor}
\label{cor4t.2}

    \begin{eqnarray}
    \label{4t.6}
\inf_{m\in [-1,1]^{\La}\cap\mathcal A}  F_{{\rm sp}(\Ga),\ga}( m | \bar m) &\ge&
\Phi_{\Delta_0} +\Phi_{\Delta_{\rm in}}(\bar m_{\si_{\rm ext}})
+\sum_k \Phi_{\Delta^+_k}(\bar m_{\si_{I_k^{+}}})\nn
\\&& \hskip1cm
+\sum_k \Phi_{\Delta^-_k}(\bar m_{\si_{I_k^{-}}}),
     \end{eqnarray}
where
    \begin{eqnarray}
    \label{4t.7}
&&\Phi_{\Delta_0} = \inf \Big\{ F^*_{\Delta_0,\ga}( m)\;\Big|\; m\in [-1,1]^{\Delta_{0}},
 |m - m^{(\ell_-)}| \le c \ga^\alpha,\;
\eta(\cdot;m)=\eta_\Ga(\cdot),
\Big\},\nn\\
&& \Phi^+_{\Delta_{\rm in}}(\bar m_{\si_{\rm ext}})=  \inf \Big\{ F'_{\Delta_{\rm in},\ga}( m|\bar m_{\si_{\rm ext}})\;\Big|\;
  m\in [-1,1]^{\Delta_{\rm in}}, \eta(\cdot,m)=1,\Big\},
\nn\\
&&
\Phi^+_{\Delta^+_k}(\bar m_{\si_{I_k^{+}}}) = \inf \Big\{F'_{\Delta^+_k,\ga}
( m|\bar m_{\si_{I_k^{+}}})\;\Big|\; m\in [-1,1]^{\Delta^+_k}, \eta(\cdot,m)=1 \Big\},
\nn\\
&&
\Phi^-_{\Delta^-_k}(\bar m_{\si_{I_k^{-}}}) = \inf \Big\{F'_{\Delta^-_k,\ga}
( m|\bar m_{\si_{I_k^{-}}})\;\Big|\;m\in [-1,1]^{\Delta^-_k}, \eta(\cdot,m)=-1 \Big\}.
     \end{eqnarray}
\end{cor}

Corollary \ref{cor4t.2} is useful for us because it allows to split the original variational
problem on the left hand side of \eqref{4t.6} into separated, localized variational problems, as
on the right hand side of \eqref{4t.6}.

Recalling \eqref{2s.15}, we have the following upper bound for the partition function in the numerator of \eqref{2.6}:
    \begin{eqnarray}
    \label{4t.8}
 e^{- \Phi_{\Delta_0} +  c |\La| \ga^{1/2} \log \ga^{-1}}
 e^{- \Phi_{\Delta_{\rm in}}(\bar m_{\si_{\rm ext}})}
\{\prod Z^+ (I^+_k)\} \{\prod Z^- (I^-_k)\},
     \end{eqnarray}
where
    \begin{equation}
    \label{4t.9}
    Z^+ (I^+_k)= \sum^{+}_{\si_{I_k^{+}}} e^{-H(\si_{I_k^{+}}) -
    \Phi^{+}_{\Delta^{+}_k}(\bar m_{\si_{I_k^{+}}})}.
    \end{equation}
The superscript $+$ in the sum means that the  sum is
restricted to spin configurations in $I_k^{+}$ such that
a configuration made by $\si_{I_k^{+}}$ in $I_k^{+}$ and by
any configuration with $\eta= 1$ in $\Delta^{+}_k$ has $\Theta= 1$ on $\partial^{+}_{{\rm out},k}(\Ga)$, see \eqref{2.6}.
$Z^- (I^-_k)$ is defined analogously.

Following the Peierls strategy we use at this point
the
spin flip symmetry to rewrite \eqref{4t.8} in a more convenient way.  In fact we have:
    \begin{equation}
    \label{4t.10}
\Phi^-_{\Delta^-_k}(\bar m_{\si_{I_k^{-}}}) = \Phi^+_{\Delta^-_k}(\bar m_{-\si_{I_k^{-}}})
    \end{equation}
and therefore  $ Z^- (I^-_k)=    Z^+ (I^-_k)$.
The numerator in \eqref{2.6} is thus bounded by
       \begin{eqnarray}
    \label{4t.11}
  && \hskip-2cm Z_{c(\Ga);\bar\si}(\eta  = \eta_\Ga \;{\rm on}\;
 \, {\rm sp}(\Ga); \Theta = \pm 1 \;{\rm on \;each} \,\; \partial_{{\rm out},k}^{\pm}(\Ga))\nn\\
&\le & e^{- \Phi_{\Delta_0} + c  |\La| \ga^{1/2} \log \ga^{-1}} \nn\\
&\times&  e^{- \Phi_{\Delta_{\rm in}}(\bar m_{\si_{\rm ext}})}
\{\prod Z^+ (I^+_k)\} \{\prod Z^+ (I^-_k)\}.
     \end{eqnarray}
The key point   is now to prove a lower bound on    $\Phi_{\Delta_0}$
so good as to kill
the error terms in the first exponent and to give
what is required by the Peierls bounds.
The other factors in \eqref{4t.11} will simplify with those coming
from the lower bound modulo a small error.  Preliminary to that is the analysis of the two layers free energy
$\hat f_\eps(m_1,m_2)$.
In Appendix \ref{sec:8s} it is proved that:

   \medskip

\begin{prop}
\label{thm4s.1} 
For any $\eps>0$ small enough
   \begin{equation}
    \label{4s.1}
\hat f_\eps(m_1,m_2) := -\frac 12 \Big( m_1^2+m_2^2\Big) + \hat\phi_\eps(m_1,m_2)
    \end{equation}
has two minimizers, $\pm m^{(\eps)}:=\pm (m_\eps,m_\eps)$, and there is a constant $c$ so that
    \begin{equation}
    \label{4s.2}
 | m_\eps - \sqrt{3 \eps}| \le c  \eps^{3/2}.
    \end{equation}
Moreover, calling $\hat f_{\eps,{\rm eq}}$ the minimum of
$\hat f_\eps(m)$, for any $\zeta>0$ small enough:
    \begin{equation}
    \label{4.19.-1}
    \Big| \hat f_\eps(m) - \hat f_{\eps,{\rm eq}}\Big|  \ge c \zeta^2,\quad \text{for all $m \notin U_\zeta$, }
    \end{equation}
where
\[
U_\zeta:= \Big\{(m_1,m_2): |m_i- m_\eps| < \frac \zeta 2,\;i=1,2
\Big\} \cup \Big\{(m_1,m_2): |m_i+ m_\eps| < \frac \zeta 2,\;i=1,2
\Big\}.
\]
\end{prop}

\medskip
The following lower bound for
$\Phi_{\Delta_0} $ follows from
Proposition \ref{thm4s.1} and proves that
the excess free energy $\Phi_{\Delta_0} -  \hat f_{\eps,{\rm eq}} \frac{|\Delta_0|}2$
grows at least like $c |\Delta_0| \ga^{4\alpha+2a}$
(recall that $\hat f_{\eps,{\rm eq}}$ is defined after \eqref{4s.2}
and that $ F^*_{\Delta_0,\ga}( m) =\hat f_{\eps,{\rm eq}} \frac{|\Delta_0|}2$
when $m$ is identically equal to $m_\eps$ or to $-m_\eps$).  As desired
such excess free energy  $c |\Delta_0| \ga^{4\alpha+2a}$
is much larger (for small $\ga$) than
the error term in the first exponent in \eqref{4t.11} which is given by
$c |{\rm sp}(\Ga)| \ga^{1/2} \log \ga^{-1}$.

   \medskip

\begin{thm}
\label{thm4s.2}   
There is $c>0$ so that
   \begin{equation}
    \label{4s.4}
\Phi_{\Delta_0} \ge\;  \hat f_{\eps,{\rm eq}} \frac{|\Delta_0|}2 + c \frac{|\Delta_0|}{ \ga^{-(1+\alpha) }\ga^{-\alpha}} \ga^{-(1-\alpha) }\min\{\ga^{\alpha}; \ga^{2a}\}.
    \end{equation}

\end{thm}

\medskip

\begin{proof}
We rewrite \eqref{2s.19} as
       \begin{eqnarray}
    \label{2s.19.1}
F^*_{\Delta_0,\ga}( m) &=& \frac 12\sum_{(x,i)\in \Delta_0} \hat f_\eps\big(m(x,i),m(v_{x,i})\big)\nn\\
&+& \frac 14\sum_{(x,i)\ne(y,i) \in \Delta_0} J_\ga(x,y)\Big (m(x,i)-m(y,i)
\Big)^2
    \end{eqnarray}
and start by bounding from below the first term.
We distinguish two cases:

\begin{itemize}
\item (i) when $\eta(x,i)=\eta(v_{x,i})\ne 0$ we bound $\hat f_\eps\big(m(x,i),m(v_{x,i})\big)
\ge \hat f_{\eps,{\rm eq}}$;

\item (ii) in the other cases for all $(y,i) \in C_x^{\ell_-,i}$,
$(m(y,i),m(v_{x,i})) \notin U_\zeta$ for $\ga$ small enough
(as, by the smoothness condition, $|m - m^{(\ell_-)}| \le c \ga^\alpha$).
We then bound
 $\hat f_\eps\big(m(x,i),m(v_{x,i})\big)
\ge \hat f_{\eps,{\rm eq}} + c \zeta^2$.
\end{itemize}

Thus the first term on the right hand side of \eqref{2s.19.1} is bounded from below by
\[
 \hat f_{\eps,{\rm eq}} \frac{|\Delta_0|}2 +   N_{(ii)} c\ga^{-(1-\alpha)} \zeta^2
\]
where, writing  $v_{x,i}=(x,i^\prime)$, $N_{(ii)}$ is the number of distinct pairs of intervals $C^{\ell_-,i}_x$,  $C^{\ell_-,i'}_x$
where case (ii) occurs.

The second term on the right hand side of \eqref{2s.19.1} is bounded from below by
retaining only the terms where $(x,i)$ and $(y,i)$ are in two consecutive
$C^{\ell_-,i}$ intervals and $\eta(x,i;m)= -\eta(y,i;m)\ne 0$.  Suppose for instance
$\eta(x,i;m)=1$, then $|m(x,i;m)-m_\eps| \le 2\zeta$ (for $\ga$ small enough and using
smoothness as before). Analogously $|m(y,i)+m_\eps| \le 2\zeta$ and recalling the
assumption that $J(\cdot)$ is strictly positive at the origin, we see that
the contribution  to \eqref{2s.19.1}  coming from any such pair of intervals is, for $\ga$ small enough, at least
\[
 \; c  \ga \frac{J(0)}2 (\ga^{-(1-\alpha)})^2 = \tilde c \ga^{-(1-\alpha)} \ga^{\alpha}.
\]

To conclude we observe that by definition of contours for  any $Q$--rectangle $Q'$ in
$\Delta_0$ there is a  rectangle $Q$ in $\Delta_0$ connected to $Q'$ (or just $Q'$ itself)
with the following property.
Either case (ii) occurs in $Q$ or there are two consecutive
$C^{\ell_-,i}$ intervals one at least inside $Q$  with opposite values of $\eta$,
or both the above events occur.
%
\eqref{4s.4} is then obtained
because the number of $Q$ rectangles in
$\Delta_0$ is $\dis{\frac{|\Delta_0|}{ \ga^{-(1+\alpha) }\ga^{-\alpha}}}$.
\end{proof}

\vskip1cm

\setcounter{equation}{0}

\section{Characterization of minimizers}
\label{sec:4s}

The lower bound of the denominator in \eqref{2.6} will be obtained
by computing the free energy functional  on a suitable test function
$m$ on ${\rm sp}(\Ga)$.


On $\Delta_0$,  $m$ will be a constant approximately equal to $m_\eps$ while it will be (approximately) equal to
the minimizers of $\Phi_{\Delta_{\rm in}}(\bar m_{\si_{\rm ext}})$, $\Phi^+_{\Delta^+_k}(\bar m_{\si_{I_k^{+}}})$,
and $\Phi^+_{\Delta^-_k}(\bar m_{-\si_{I_k^{-}}})$, in the
respective sets $\Delta_{\rm in}$, $\Delta^+_k$, and  $\Delta^-_k$.

The main difficulty will be to estimate
the last term in \eqref{4s.10} which in the upper bound for the partition function
could be neglected being non negative.
We shall prove below that the term $[m(x,i) -m(y,i)]^2$
in \eqref{4s.10}
with $(x,i)\in \Delta_0$ and $(y,i)\notin \Delta_0$
is bounded from above by $e^{-c\ga^{-\alpha}}$, with the above choices of $m(x,i)$
and  $m(y,i)$. Thus the last term in \eqref{4s.10} will then
be negligible also in the lower bound.

The analysis of the minimizers is essentially the same for all of them and, for the sake of definiteness, we will just
look at the minimizer of $\Phi^+_{\Delta_{\rm in}}(\bar m_{\si_{\rm ext}})$, referring to \eqref{4t.7} and \eqref{2s.20} for the definition.

Recalling \eqref{2s.20}, \eqref{2s.14} and \eqref{4s.5.23}
we have
\begin{eqnarray}
    \label{4s.5.22}
&&F'_{\Delta_{\rm in},\ga}( m|\bar m)  = \sum_{(x,i)\in
\Delta_{\rm in}}\frac 12 \Big(\hat \phi_{\eps} \big(m(x,i),m(v_{x,i})\big) -a_{x,i}m(x,i)^2\Big) \nonumber\\&&\hskip.5cm -  \sum_{(x,i)\in
\Delta_{\rm in}} m(x,i) \Big(\frac 12
 \sum_{(y,i)\in \Delta_{\rm in}}
J_\ga(x,y) (m(y,i) +\sum_{(y,i)\notin  {\rm sp}(\Ga)}
J_\ga(x,y) \bar m(y,i) \Big).
    \end{eqnarray}

Being a continuous function of $\{m(x,i), (x,i)\in \Delta_{\rm in}\}$
the function
$m \mapsto F'_{\Delta_{\rm in},\ga}( m|\bar m)$ attains a minimum
when $m$ varies in the compact set
   \begin{equation}
    \label{4s.5.23.00}
K:=\bigcap_{(x,i)\in \Delta_{\rm in}} \Big\{|m^{\ell_-}(x,i)-m_\eps|\le \zeta
\Big\}.
   \end{equation}
We are going to prove that the minimizer is unique and will establish properties of
the minimizer typical of the correlations in the Gibbsian high temperatures regime.

We fix arbitrarily a pair $(x,i)$ and $(x,i')=v_{x,i}$
of vertically interacting sites in $\Delta_{\rm in}$, and  regard
$F'_{\Delta_{\rm in},\ga}( m|\bar m) $ in \eqref{4s.5.22}
as a function of $m(x,i)$ and $m(x,i')$ alone with all the other $m(y,j)$ considered as
fixed parameters that we denote  by $u(y,j)$.  Let
   \begin{eqnarray}
    \label{9s.8}
 \mathcal N_{x,i,i'} &=& \Big\{ u(y,j),  y \ne x, j=i,i' : (y,j) \in \Delta_0^c,\; u(y,j) \in (-1,1)\; \text{ and}\nn\\
&&\;\;\; | u^{\ell_-}(y,j) - m_\eps| \le \zeta,  \quad \text{for all $(y,j) \notin C^{\ell_-,j}_x$}
\Big\},
    \end{eqnarray}
namely the set where the function $\eta(\cdot;u)$ 
is identically 1 except maybe on the intervals containing $(x,i)$ or $(x,i')$ where we do not impose conditions on the $u(\cdot)$.

For any $u\in  \mathcal N_{x,i,i'}$ we introduce the function
   \begin{equation}
    \label{9s.5}
g_\eps(m_i,m_{i'}) :=    \hat \phi_{\eps}(m_i,m_{i'})-\frac 12 (a_im_i^2+a_{i'}m_{i'}^2)
- \la^u_i m_i - \la^u_{i'}m_{i'}
    \end{equation}
where $(m_i,m_{i'}) \in (-1,1)\times (-1,1)$, $a_j$ is a shorthand for $a_{x,j}$, $j=i,i'$,
and
   \begin{eqnarray}
    \label{9s.7}
&&
\la^u_j =  \sum_{y\ne x:(y,j)\notin \Delta_{0}}
J_\ga(x,y) u(y,j),\quad j = i,i'.
    \end{eqnarray}
In Appendix \ref{sec:9s} we shall prove:

   \medskip

\begin{prop}
\label{thm4s.3.1} 
There are $\eps_0>0$ and $\ga_\cdot:(0,\infty)\to(0,\infty)$ such that for any $0<\eps\le \eps_0$
there are $r<1$ and coefficients $C_{x,i,i'}(j,j')$, $j,j' \in\{i,i'\}$, so
that the following holds for all $\ga\le \ga_\eps$.

\begin{itemize}
\item  For any $u\in \mathcal N_{x,i,i'}$
there is a unique minimizer $m^{(u)}=(m^{(u)}_i,m^{(u)}_{i'})$ of $g_\eps$.

\item $\dis{ \sum_{j'=i,i'} C_{x,i,i'}(j,j') \le r}$ for $j=i,i'$.

\item $\dis{|m^{(u)}_j-m_\eps| \le \sum_{j'=i,i'} C_{x,i,i'}(j,j')
\frac
{|\la^{u}_{j'} - \la_{j'}^{\rm eq}|}
{1-a_{x,j'}}}$,  $j=i,i'$, where $\la_j^{\rm eq}$ is the value of $\la_j$ when $u$ is identically equal to $m_\eps$.

    \item
    $\dis{ \frac
{| \la_j^u -\la_j^{\rm eq}|} {1-a_{x,j}}\le  \zeta + c \ga^\alpha}$  for $j=i,i'$

\item For any $u,v\in \mathcal N_{x,i,i'}$ and  $j=i,i'$,
$\dis{|m^{(u)}_j-m^{(v)}_j| \le \sum_{j'=i,i'} C_{x,i,i'}(j,j')\frac
{|\la^{u}_{x,j'} - \la_{x,j'}^{v}|}
{1-a_{x,j'}}}$.
\end{itemize}
\end{prop}

As a consequence of the above proposition we have:

\medskip

\begin{thm}
\label{thm4s.3}
In the same context of Proposition \ref{thm4s.3.1} and for $\ga$ small enough
the following holds.
There is
a unique minimizer $m^*$ of $F'_{\Delta_{\rm in},\ga}( \cdot|\bar m)$ in $K$,
see \eqref{4s.5.23.00} and for any $(x,i)\in \Delta_{\rm in}$
  \begin{eqnarray}
    \label{4s.5.25.3}
&&  |m^*(x,i)-m_\eps| <\zeta, \\
&&  |m^*(x,i)-m_\eps| < 2r^n,
\label{4s.5.25.3.1}
    \end{eqnarray}
where $n$ is the minimal number of steps required to go from $(x,i)$ to the
complement of
${\rm sp}(\Ga)$ when horizontal steps have length $\le \ga^{-1}$ while
the vertical steps have length 1.

\end{thm}

\medskip
\begin{proof}
We shall
preliminary prove that for $\ga$ small enough the minimizer
in  Proposition \ref{thm4s.3.1} satisfies $|m^{(u)}_j-m_\eps|<\zeta $.
Indeed:
   \begin{equation}
    \label{4s.5.333}
|m^{(u)}_j-m_\eps| \le \sum_{j'=i,i'} C_{x,i,i'}(j,j')\frac
{|\la^{u}_{x,j'} - \la_{j'}^{\rm eq}|}
{1-a_{j'}}
\le r \zeta + c \ga^\alpha < \zeta,
    \end{equation}
having used the bounds on $C_{x,i,i'}(j,j')$ and $|\la^{u}_{j'} - \la_{j'}^{\rm eq}|$ stated
in  Proposition \ref{thm4s.3.1}. The last inequality  $ r \zeta + c \ga^\alpha < \zeta$
holds for $\ga$ small enough, because
$r<1$ and by the choice of $\zeta$ and $\alpha$.

Since $F'_{\Delta_{\rm in},\ga}( m|\bar m)$ is a continuous function
of the coordinates $m(x,i)$, $(x,i)\in \Delta_{\rm in}$, it has a minimum in the compact set $K$.
Let $m$ be a minimizer, and $C^{\ell_-,i}_{x}$  a segment in $\Delta_{\rm in}$
whose points are denoted
$(x_1,i),..,(x_N,i)$. Let $m_{x_1;i}$ be the function obtained from
$m$ after replacing the elements $m(x_1,i)$ and $m(v_{x_1,i})$ by the minimizer
of $g_\eps$ relative to the points $(x_1,i)$ and $v_{x_1,i}$ and with $u=m$ on the complement
of $\{(x_1,i),v_{x_1,i}\}$.  We then define iteratively the sequence
$m_{x_1,\dots,x_k;i}$, $k\le N$, by applying the above procedure to
$m_{x_1,\dots,x_{k-1};i}$.  We claim that $m_{x_1,\dots,x_{N};i}=m$.  In fact
$F'_{\Delta_{\rm in},\ga}(m_{x_1,\dots,x_k;i}|\bar m)$ is non increasing in $k$ (because
we are relaxing the condition $\eta=1$ in $C^{\ell_-,i}_{x,i}$ and because
we are putting at each step the minimizer of the corresponding $g_\eps$) and therefore
\[
F'_{\Delta_{\rm in},\ga}(m_{x_1,\dots,x_N;i}|\bar m) \le F'_{\Delta_{\rm in},\ga}( m|\bar m).
\]
By \eqref{4s.5.333}
$|m_{x_1,\dots,x_{N};i}(y,j) - m_\eps| < \zeta$ for all $(y,j)$ in
$C^{\ell_-,i}_{x,i}\cup C^{\ell_-,i}_{v_{x,i}}$. Therefore
$m_{x_1,\dots,x_{N};i} \in K$ and since $m$ is a minimizer
\[
F'_{\Delta_{\rm in},\ga}( m|\bar m) \le F'_{\Delta_{\rm in},\ga}(m_{x_1,\dots,x_N;i}|\bar m).
\]
This  means that at each step
\[
F'_{\Delta_{\rm in},\ga}(m_{x_1,\dots,x_{k-1};i}|\bar m) =  F'_{\Delta_{\rm in},\ga}(m_{x_1,\dots,x_k;i}|\bar m)
\]
and by the uniqueness of the minimizer of $g_\eps$ we conclude the proof of the claim.
Observe that we have also proved a sort of DLR property, namely
that if $m$ is a
minimizer then its values at $(x,i)$ and $(x,i')$ minimize
the corresponding $g_\eps$.

By the arbitrariness of the choice of $(x,i) \in \Delta_{\rm in}$ in the above argument we deduce that for all $(x,i) \in \Delta_{\rm in}$
$|m(x,i) - m_\eps| < \zeta$, and since  $(m(x,i),m(v_{x,i}))$ is the minimizer of the corresponding $g_\eps$ then by the second property
in Proposition \ref{thm4s.3.1} and by  \eqref{9s.7}
   \begin{equation}
    \label{4s.5.334}
|m(x,i)-m_\eps| \le \sum_{j'=i,i'} C_{x,i,i'}(i,j')
\sum_{y\ne x:(y,j)\notin \Delta_{0}}
\frac
{J_\ga(x,y)}
{1-a_{x,j'}}  |m(y,j') - m_\eps|,
    \end{equation}
where $m=\bar m$ outside ${\rm sp}(\Ga)$.  The inequality \eqref{4s.5.334} can be iterated $n$ times with
$n$ as in the text of the theorem and we get
   \begin{eqnarray*}
|m(x,i)-m_\eps| &\le& \sum_{j_1=i,i'} C_{x,i,i'}(i,j_1)
\sum_{y_1\ne x:(y,j)\notin \Delta_{0}}
\frac
{J_\ga(x,y_1)}
{1-a_{x,j_1}}  \\ &\times&  \sum_{j_2=j_1,j_1'} C_{y_1,j_1,j_1'}(j_1,j_2)
\sum_{y_2\ne y_1:(y_2,j_2)\notin \Delta_{0}}
\frac
{J_\ga(x,y_2)}
{1-a_{y_1,j_2}} \cdots |m(y_n,j_n)-m_\eps)|.
    \end{eqnarray*}
We bound the last factor by $2$, the sum over the $y_k$ is normalized to 1
hence $|m(x,i)-m_\eps| \le 2r^n$.

It remains to prove the uniqueness of the minimizer  of $F'_{\Delta_{\rm in},\ga}( \cdot|\bar m)$ in $K$.  Suppose there are two minimizers $m$ and $m'$, then by the last statement
of Proposition \ref{thm4s.3.1}
   \begin{equation}
    \label{4s.5.335}
|m(x,i)-m'(x,i)| \le \sum_{j'=i,i'} C_{x,i,i'}(i,j')
\sum_{y\ne x:(y,j)\notin \Delta_{0}}
\frac
{J_\ga(x,y)}
{1-a_{x,j'}}  |m(y,j') - m'(y,j')|.
    \end{equation}
The inequality can be iterated $n$ times. But now $n$ is arbitrary
because $m(y,j) = m'(y,j)$ outside ${\rm sp}
(\Ga)$ and therefore $m(x,i)=m'(x,i)$.

%
%
%
%
%

\end{proof}

\vskip1cm

\setcounter{equation}{0}

\section{The lower bound}
\label{sec:6t}
In this section we will prove a lower bound for the denominator in \eqref{2.6}.

We call  \emph{trial function} a function $m\in
\mathcal M_{{\rm sp}(\Ga)}$, see \eqref {2s.13}, namely with values in
$M_{\ga^{-1/2}}$ and constant on the intervals
$C^{\ga^{-1/2},i}$ contained in ${\rm sp}(\Ga)$.  Denote by $\si$ the collection
of spins in $c(\Ga)^c$ and in the sets $I_k^{\pm}$. For any such $\si$ we choose
a trial function $m_{\si}$ and  using \eqref{2s.16} we get that the denominator in \eqref{2.6} is bounded from below by
    \begin{equation}
    \label{6t.1}
 e^{-c |{\rm sp}(\Ga)| \ga^{1/2-a}}\sum^{+}_{\si_{I_k^{+}},\,\si_{I_k^{-}}}
 e^{-F_{{\rm sp}(\Ga),\ga}( m_\si | \bar m_\si)- \sum_k H(\si_{I_k^{+}})
 - \sum_k H(\si_{I_k^{-}})},
    \end{equation}
having used the notation of \eqref{4t.9}.  The whole point is now to reduce
\eqref{6t.1} to \eqref{4t.11} and this is done with a good choice of the
trial function.
We fix $\si$ and start with the function $m'$ which in $\Delta_0$ is identically equal to $m_\eps$,
in $\Delta_{\rm in}$  it is the minimizer of $\Phi^+_{\Delta_{\rm in}}(\bar m_{\si_{\rm ext}})$
and in $I_k^{\pm}$ it coincides with the minimizer of $\Phi^+_{\Delta^{\pm}_k}(\bar m_{\si})$.  Since the values
of $m'$ are not necessarily in $M_{\ga^{-1/2}}$ $m'$ may not be a trial function.
We then define $m''$  which at each $(x,i)$ is equal to
a value in $M_{\ga^{-1/2}}$ which minimizes the distance of
$m'(x,i)$ from $M_{\ga^{-1/2}}$.  $m''$ may not be constant on the intervals $C^{\ga^{-1/2},i}$ so that we define $m_\si$ as
   \begin{equation}
    \label{6t.2}
m_\si(x,i) = \ga^{1/2}\sum_{y \in C^{\ga^{-1/2}}_{x}}m''(y,i).
    \end{equation}
$m_\si$ is a trial function and we can use it in \eqref{6t.1}.  We claim that
   \begin{equation}
    \label{6t.3}
 F_{{\rm sp}(\Ga),\ga}( m_\si | \bar m_\si) \le F_{{\rm sp}(\Ga),\ga}( m'' | \bar m_\si)
 + c \ga^{1/2}|{\rm sp}(\Ga)|.
    \end{equation}
{\em Proof.}
Recalling the definition \eqref{2s.14} of the free energy functional, we observe
that by convexity
        \begin{eqnarray*}
\sum_{(x,i)\in {\rm sp}(\Ga)}
\hat \phi_{\eps} (m''(x,i),m''(v_{x,i}))\ge \sum_{(x,i)\in {\rm sp}(\Ga)}
\hat \phi_{\eps} (m_\si(x,i),m_\si(v_{x,i})).
    \end{eqnarray*}

  For the terms in \eqref{2s.14} that contain $J_\gamma$, replacing $m_\si$ by $m''$ gives an error
  that is bounded from above by $c\ga^{1/2}|{\rm sp}(\Ga)|$, yielding \eqref{6t.3}.
%
%
Similarly
   \begin{equation}
    \label{6t.4}
 F_{{\rm sp}(\Ga),\ga}( m'' | \bar m_\si) \le F_{{\rm sp}(\Ga),\ga}( m' | \bar m_\si)
 + c \ga^{1/2}|{\rm sp}(\Ga)|.
    \end{equation}
So far we have proved that
the denominator in \eqref{2.6} is bounded from below by
    \begin{equation}
    \label{6t.5}
 e^{-c |{\rm sp}(\Ga)| \ga^{1/2-a}}\sum^{+}_{\si_{I_k^{+}},\,\si_{I_k^{-}}}
 e^{-F_{{\rm sp}(\Ga),\ga}( m' | \bar m_\si)- \sum_k H(\si_{I_k^{+}})
 - \sum_k H(\si_{I_k^{-}})}
    \end{equation}
(with $c$ a suitable constant which takes care of all the above errors).
We next use \eqref{4s.10} to get
    \begin{eqnarray}
    \label{6t.6}
F_{{\rm sp}(\Ga),\ga}( m' | \bar m_\si) &\le&  \hat f_{\eps,{\rm eq}} \frac{|\Delta_0|}2 +
  \Phi^+_{\Delta_{\rm in}}(\bar m_{\si})\nn\\ &+&
 \sum_{k}\Phi^+_{\Delta^+_k}(\bar m_{\si})
  +  \sum_{k}\Phi^+_{\Delta^-_k}(\bar m_{\si})
  + c |{\rm sp}(\Ga)| e^{-\ga^{-\alpha}},
     \end{eqnarray}
where we used that (i) $m'=m_\eps$ in $\Delta_0$; (ii) $m'$ is the minimizer
of   $\Phi^+_{\Delta^{\pm}_k}(\bar m_{\si})$ and of $\Phi^+_{\Delta_{\rm in}}(\bar m_{\si})$ in the
respective sets;
(iii) the last term in \eqref{4s.10} is bounded using \eqref{4s.5.25.3.1}.

 In conclusion
        \begin{eqnarray}
    \label{6t.7}
  && \hskip-2cm Z_{c(\Ga);\bar\si}(\eta  = 1 \;{\rm on}\;
 \, {\rm sp}(\Ga); \Theta = \pm 1 \;{\rm on \;each} \,\; \partial_{k}^{\pm}(\Ga))\nn\\
&\ge & e^{-  \hat f_{\eps,{\rm eq}} \frac{|\Delta_0|}2 -c (|{\rm sp}(\Ga)|\ga^{1/2}} \nn\\
&\times&  e^{- \Phi_{\Delta_{\rm in}}(\bar m_{\si_{\rm ext}})}
\{\prod Z^+ (I^+_k)\} \{\prod Z^+ (I^-_k)\}.
     \end{eqnarray}

\noindent \emph{Proof of Theorem \ref{peierls}.}
A comparison with the upper bound  \eqref{4t.11} and use of \eqref {4s.4} then completes
the proof of the Peierls bounds.\qed

%
%

\vskip1cm

\setcounter{equation}{0}

\section{Proof of Theorem \ref{thm1.1}}
\label{sec:7t}

The proof of Theorem \ref{thm1.1} is based
on the validity of the Peierls bounds \eqref{2.66} and it follows closely the
well known proof for the nearest neighbor Ising model at low temperatures.

Let $\{\La_n\}$ be an increasing sequence of bounded $Q$-measurable regions
which invades the whole space and $\mu^+_{\La_n,\bar \si}$ plus diluted Gibbs measures
with boundary conditions $\bar \si$ ($\bar\si$ may depend on $n$).  We want to prove that for $\ga$ small enough
and all boundary conditions $\bar \si$ as in the paragraph of~(\ref{2.6})
   \begin{equation}
    \label{7t.1}
 \lim_{n\to \infty} \mu^+_{\La_n,\bar \si}\Big[ \Theta(0) < 1\Big] < \frac 12.
    \end{equation}
By the definition of plus diluted Gibbs measures the event in \eqref{7t.1} can only occur
if there is a contour $\Ga$ such that the origin belongs to $c(\Ga)$.  Call $N(\Ga)$ the number of $Q$-rectangles contained in ${\rm sp}(\Ga)$.  Then there is a
horizontal translate of ${\rm sp}(\Ga)$ by $k\ell_+$, $k\le N(\Ga)$, so that the translate
of ${\rm sp}(\Ga)$  contains the origin.

This means that
   \begin{equation}
    \label{t1.1}
\mu^+_{\La_n,\bar \si}\Big[ \Theta(0) < 1\Big]
\le \mu^+_{\La_n,\bar \si}\Big[ \bigcup_{\Ga:{\rm sp}(\Ga) \ni 0}\bigcup_{k\le N(\Ga)}\{\Ga_k\mbox{ is a contour}\}\Big],
    \end{equation}
where $\Ga_k$ is the contour obtained from $\Ga$ by translating it by $k\ell^+$. By subadditivity, the right hand side of~(\ref{t1.1})
is bounded above by
\begin{equation}
    \label{t1.2}
    \sum_{\Ga:{\rm sp}(\Ga) \ni 0}\sum_{k\le N(\Ga)}
    \mu^+_{\La_n,\bar \si}\Big[ \Ga_k\mbox{ is a contour}\Big].
    \end{equation}
Now the probability inside the double sum in~(\ref{t1.2}) equals
\begin{equation}
    \label{t1.3}
 \frac
{Z_{c(\Ga_k);\bar\si}(\eta  = \eta_{\Ga_k} \;{\rm on}\;
 \, {\rm sp}(\Ga_k); \Theta = \pm 1 \;{\rm on \;each} \,\; \partial_{{\rm out},i}^{\pm}(\Ga_k))}
{Z^+_{\La_n,\bar \si}}\le W_{\Ga_k}(\bar\si),
 \end{equation}
the inequality justified by the fact that the denominator in~(\ref{2.6}) is bounded above $Z^+_{\La_n,\bar \si}$
(since the sum defining the latter quantity contains the terms in the sum defining the former one).

Therefore by  \eqref{2.66}, and using the fact that $|{\rm sp}(\Ga_k)|=|{\rm sp}(\Ga)|$, we find that
    \begin{equation}
    \label{7t.2}
\mu^+_{\La_n,\bar \si}\Big[ \Theta(0) < 1\Big] \le \sum_{\Ga:{\rm sp}(\Ga) \ni 0}N(\Ga)  e^{-c |{\rm sp}(\Ga)| \ga^{2a+4\alpha}}.
    \end{equation}
On the other hand $|{\rm sp}(\Ga)| = N(\Ga) \ga^{-(1+\alpha)}
\ga^{- \alpha}$ so that the sum on the right hand side of \eqref{7t.2} is just the sum over all 
connected regions $D$ made of unit cubes of
    \begin{equation}
    \label{7t.3}
\mu^+_{\La_n,\bar \si}\Big[ \Theta(0) < 1\Big] \le \sum_{D \ni 0}
|D|  e^{-c |D| \ga^{-1 +2a+2\alpha}}.
    \end{equation}
Since $a$ and $\alpha$ are much smaller than 1, then
the sum vanishes in the limit when $\ga\to 0$, see for instance Lemma 3.1.2.4
in \cite{presutti}, so that \eqref{7t.1} is proved.

By \eqref{7t.1} and the spin flip symmetry it follows that there are at least two
DLR measures, hence by ferromagnetic inequalities the plus and minus
DLR measures $\mu^{\pm}_\ga$ of  Theorem \ref{thm1.1} are distinct and
Theorem \ref{thm1.1} is proved.  There are many more consequences of the Peierls bounds, see for instance
Chapter 12 in  \cite{presutti}, but we shall not discuss such extensions here.

\setcounter{equation}{0}

\begin{appendix}

\section{Proof of Proposition \ref{thm3s.2}}
\label{sec:5s}

We prove Proposition \ref{thm3s.2} via equivalence of ensembles.
The grand-canonical partition function $\pi_\eps$ of the two layers
is trivially equal to the logarithm of
    \begin{eqnarray*}
&&\sum_{(\si_1,\si_2) \in \{-1,+1\}^2}
e^{ \{h_1 \si_1 +h_2 \si_2\}+ \eps \si_1\si_2}
= 2 \{e^\eps \cosh (h_+ )+ e^{-\eps} \cosh (h_-)\},
    \end{eqnarray*}
 where $h_+ = h_1+h_2$ and $h_- = h_2-h_1$.
Thus the pressure is given by
   \begin{equation}
    \label{2s.1}
\pi_{\eps}(h_+,h_-)
=\log (2Z),\quad Z=\{e^\eps \cosh (h_+ )+ e^{-\eps} \cosh (h_-)\}.
    \end{equation}
We can easily check that the function $(h_+,h_-)\mapsto \pi_{\eps}(h_+,h_-)$ is strictly convex, namely its Hessian, denoted here
by $D^2 \pi_\eps$, is a positive definite operator.  Indeed, by computation we have:
   \begin{equation}
    \label{2s.3}
\frac{\partial}{\partial h_+}\pi_{\eps}(h_+,h_-)= e^\eps \frac{\sinh (h_+ )}{Z},
\quad
\frac{\partial }{\partial h_-}\pi_{h_+,h_-}= e^{-\eps} \frac{\sinh (h_- )}{Z},
    \end{equation}
        \begin{eqnarray*}
\frac{\partial^2 }{\partial h_+^2}\pi_{\eps}(h_+,h_-) &=&  \frac{e^\eps}{Z}\Big(  \cosh (h_+ )-  e^{\eps}\frac{\sinh^2 (h_+ )}{Z} \Big) >0,
\\
\frac{\partial^2 }{\partial h_-^2}\pi_{\eps}(h_+,h_-) &=&  \frac{e^{-\eps}}{Z}\Big(  \cosh (h_- )-  e^{-\eps}\frac{\sinh^2 (h_- )}{Z} \Big) >0,
\\
\frac{\partial^2 }{\partial h_+ \partial h_-}\pi_{\eps}(h_+,h_-) &=& - \frac{\sinh (h_+ )\sinh (h_- )}{Z^2}.
    \end{eqnarray*}
  It then follows that the diagonal elements of $D^2 \pi_\eps(h_+,h_-)$ and its determinant, given by
   \begin{eqnarray}
    \label{2s.2}
&&|D^2\pi_{\eps}(h_+,h_-)|  = Z^{-4} \Big( 1+2 \cosh (2\eps) \cosh(h_+)\cosh(h_-)
+  \cosh^2(h_+)\cosh^2(h_-)\nonumber\\&&\hskip2cm - \sinh^2(h_+)\sinh^2(h_-)\Big),
    \end{eqnarray}
   are all positive, and therefore the $2\times 2$ Hessian matrix is positive definite.
 We now consider the Legendre transform of $\pi_{\eps}$:
 \begin{equation}
    \label{2s.7.1}
 \phi_\eps(m) := \sup_{h} \Big( \frac 12
 \langle h, m\rangle -\pi_{\eps}(h) \Big)
    \end{equation}
    where  $m=(m_+,m_-)$  and
       \begin{equation}
    \label{2s.8}
  \langle h, m\rangle  := h_+m_++h_-m_-, 
    \end{equation}
and we have:


\medskip

\begin{lem}
For any $m=(m_+,m_-)$ such that  $|m_i|<1$, $i=1,2$, where
  \begin{equation}
    \label{2s.4}
m_1 = (m_+-m_-)/2,\quad m_2 = (m_++m_-)/2
      \end{equation}
the sup in \eqref{2s.7.1} is a maximum, achieved at a unique
$h=(h_+, h_-)$, which is the unique solution of
      \begin{equation}
    \label{2s.5}
 m_+=2\frac{\partial}{\partial h_+}\pi_{\eps}(h_+,h_-),
\quad m_-= 2
 \frac{\partial}{\partial h_-}\pi_{\eps}(h_+,h_-).
    \end{equation}
In other words, for this $h$
 \begin{equation}
    \label{2s.7}
 \phi_\eps(m)= \frac 12 \langle h, m\rangle -\pi_\eps(h) = \sup_{h'} \Big( \frac 12
 \langle h', m\rangle -\pi_{\eps}(h') \Big).
    \end{equation}

\end{lem}

\medskip

\begin{proof}
If $|m_i|<1$, $i=1,2$, the function
\[
\Ga(h):= \frac 12\langle h, m\rangle -\pi_{\eps}(h)
 \]
 goes to $-\infty$ when $|h|\to \infty$.
 Together with the continuity of $\Gamma (h)$ this implies that the supremum is
a maximum, achieved at the critical point,
hence \eqref{2s.5}.
Uniqueness follows from the strict convexity of $\pi_\eps$.

 \end{proof}


The following lemma is an immediate consequence of the strict convexity of $\pi_\eps$ and the properties of the Legendre transform.

\begin{lem}
\label{lemma5s.2}

$\phi_{\eps}$ is strictly convex.  Writing $D\phi_\eps$ for its gradient,
$m=(m_+,m_-)$ solves the equation $\dis{D\phi_\eps(m) = \frac{\theta}2}$, $\theta=(\theta_+,\theta_-)$, if and only if
    \begin{equation}
    \label{2s.8.1}
m = 2 D \pi_\eps (\theta).
    \end{equation}
More explicitly:
   \begin{eqnarray}
    \label{2s.8.2}
m_+ &=& 2e^\eps \frac{\sinh (\theta_+)}{e^\eps \cosh (\theta_+ )+e^{-\eps} \cosh (\theta_-  )} \nonumber
\\
\\ m_-  &=& 2e^{-\eps} \frac{\sinh (\theta_-)}{e^\eps \cosh (\theta_+)+e^{-\eps} \cosh (\theta_-)}. \nonumber
    \end{eqnarray}
\end{lem}
%
%



 \medskip

 Changing back to coordinates $(h_1,h_2)$ and $(m_1,m_2)$, let
 $\hat \pi_{\eps}(h_1,h_2)$
and $\hat \phi_{\eps}(m_1,m_2)$ denote
the functions $\pi_{\eps}(h_+,h_-)$
and $\phi_{\eps}(m_+,m_-)$ when $h_{\pm}$ and $m_{\pm}$ are expressed
in terms of $(h_1, h_2)$ and respectively $(m_1,m_2)$.  Thus $\hat \pi_{\eps}$
and $\hat \phi_{\eps}$ are the Legendre transform of each other:
    \begin{equation}
    \label{2s.6}
 \hat\phi_\eps(m_1,m_2)=  \sup_{(h'_1,h'_2)} \Big((h'_1m_1+h'_2m_2)-\hat \pi_\eps(h'_1,h'_2) \Big) =  (h_1m_1+h_2m_2)-\hat\pi_\eps(h_1,h_2)
    \end{equation}
where in the last equality $(h_1,h_2)$ are functions of $(m_1,m_2)$ via \eqref{2s.5}.

%

The following lemma is an immediate consequence of the above.

\begin{lem}
\label{lemma5s.2bis}

The map $(m_1,m_2) \mapsto \hat\phi_{\eps}(m_1,m_2)$ is strictly convex on $(-1,1)^2$ and $\hat m=(m_1,m_2)$ solves the equation
$ D \hat\phi_\eps(\hat m) =  \hat\theta $, $\hat\theta=(\hat\theta_1,\hat\theta_2)$, if and only if
    \begin{equation}
    \label{2s.8.1bis}
\hat m = D \hat\pi_\eps (\hat\theta),
    \end{equation}
    where $D\hat \pi_\eps$ and $D\hat \phi_\eps$ denote the gradient of $\hat \pi_\eps$ and $\hat \phi_\eps$.
    Moreover, calling
\[
m=(m_+,m_-), \;\; m_+ = \frac{\hat m_1+\hat m_2}2,\; m_- = \frac{\hat m_2-\hat m_1}2,
\]
\[
\theta=(\theta_+,\theta_-), \;\; \theta_+ = \frac{\hat \theta_1+\hat \theta_2}2,\; \theta_- = \frac{\hat \theta_2-\hat \theta_1}2,
\]
we have that $D \hat\phi_\eps(\hat m) = \hat \theta $ if and only if $\dis{D\phi_\eps(m) = \frac{\theta}2}$.

\end{lem}


\noindent {\bf Proof of Proposition \ref{thm3s.2}}\\
It remains to prove that $\hat \phi_\eps$
is the canonical free energy of the two layers system.

%
%
%
 With $(h_1,h_2)$ as in \eqref{2s.6} and $Z$ as in \eqref{2s.1}, we get
    \begin{equation}
    \label{2s.10}
Z_{\eps,n}(m_1,m_2) = e^{-n (h_1m_1+h_2m_2)} (2Z)^{n}P_{h_1,h_2,\eps}\Big[
{ \sum_{x=1}^n \si_i(x) = m_i n, \; i=1,2} \Big],
    \end{equation}
where $P_{h_1,h_2,\eps}$ is the Gibbs measure with Hamiltonian
\[
-{\sum_{x=1}^n \{h_1 \si_1(x) +h_2 \si_2(x)\}- \eps \sum_{x=1}^n \si_1(x)\si_2(x)}
\]
hence the upper bound in \eqref{2ss.9}.
The lower bound follows by an application of Lemma~\ref{lemma_a1} below,
which gives a(n elementary) local limit
theorem lower bound for the product measure $P_{h_1,h_2,\eps}$ of the form
const $n^{-3/2}$, with a uniform constant over $n$ and $m$. \qed

\medskip

%

\begin{lem}
\label{lemma_a1}
With notation as above (see~(\ref{2s.10})), there is a
constant $c>0$ independent of $m=(m_1,m_2)\in\{-1+\frac2n,\ldots,1-\frac2n\}^2$
and $n\geq1$ such that
\begin{equation}
 \label{loc}
P_{h_1,h_2,\eps}\!\left[\sum_{x=1}^n \si_i(x) = m_i n, \; i=1,2 \right]\geq cn^{-3/2}.
\end{equation}
\end{lem}

\noindent{\bf Remark} {\em The sharp bound is $cn^{-1}$,
but~(\ref{loc}) is enough for our purposes, and requires a shorter argument, given next.}

\medskip

\noindent{\it Proof of Lemma~\ref{lemma_a1}.}  We first note that $\sum_{x=1}^n (\si_1(x), \si_2(x))$ is a
two dimensional random walk with jumps to the nearest diagonals. We rotate by $-\pi/4$ radians
and rescale space by $1/\sqrt2$ in order to get a simple 2-d random walk, denoted by
$X=(X(n)=(X_1(n),X_2(n)))_{n\geq0}$, with mean jump $m_*=(m_+,m_-)/2$.
Notice two things:

\begin{enumerate}
\item $m_*\in\{(i,j)/n:\, i,j=-n+2,-n+3,\ldots,n-3,n-2;\, i,j$ have the same parity\};
\item $X_1(n)$ and $X_2(n)$ have the same parity for every $n\geq0$.
\end{enumerate}

In these terms we want to  prove a lower bound for
\begin{equation}
 \label{loc1}
P_{h_1,h_2,\eps}(X(n)=m_*n).
\end{equation}

Let $p_1,q_1,p_2,q_2$ denote the jump probabilities of $X$ to the right, left, up and down,
respectively. Notice that $m_+=p_1-q_1$ and $m_-=p_2-q_2$.

Let $H(n)$ denote the number of horizontal steps given by $X$ in the first
$n$ steps. Then $H(n)$ has a binomial distribution with success probability $h:=p_1+q_1$.
Let us assume that $h\leq1/2$ for the remainder of the argument. A similar reasoning holds
in the other case. Now given $H(n)=k$, we have that $(X_1(n),X_2(n))=(Y_1(k),Y_2(n-k))$,
where $Y_1,Y_2$ are independent  simple random walks in 1d with respective jump probabilities to the
right $r=p_1/(p_1+q_1)$ and $s=p_2/(p_2+q_2)$. Notice that $p_1+q_1>0$, $p_2+q_2>0$.

We thus have that for any $k=0,1,\ldots,n$
\begin{eqnarray}\nonumber
&&P_{h_1,h_2,\eps}(X(n)=m_*n)\geq P_{h_1,h_2,\eps}(X(n)=m_*n,\,H(n)=k)\\
 \label{loc2}
&=&P_{h_1,h_2,\eps}(Y_1(k)=m_+n)
                           P_{h_1,h_2,\eps}(Y_2(n-k)=m_-n)
                           P_{h_1,h_2,\eps}(H(n)=k).
\end{eqnarray}

Let now $Y'_1(k)=(Y_1(k)+k)/2$. Then $Y'_1(k)$ is binomial with parameters $k$ and $r$, and the first
probability in~(\ref{loc2}) equals
\begin{equation}
 \label{loc3}
P_{h_1,h_2,\eps}(Y'_1(k)=(m_+n+k)/2).
\end{equation}
Similarly, the second probability in~(\ref{loc2}) equals
\begin{equation}
 \label{loc4}
P_{h_1,h_2,\eps}(Y'_2(n-k)=(m_-n+n-k)/2),
\end{equation}
where $Y'_2(n-k)$ is binomial with parameters $n-k$ and $s$.

We will now choose $k=k_n$ either $\lf hn\rf$ or $\lf hn\rf+1$ so that $k_n$ has the same parity as $m_+n$.
Notice that in this case $k_n\geq1$ and $(m_+n+k_n)/2$ is an integer.

A straightforward recourse to Stirling shows that the last probability in~(\ref{loc2}) is bounded from below by
a constant times $1/\sqrt n$. We next argue that the same holds for the probabilities in~(\ref{loc3}) and~(\ref{loc4}) ,
and we will be done. We consider the first such probability; the second one can be similarly treated.

We have that  $k_n=(p_1+q_1)n+\theta_n$, where $\theta_n\in(-1,1]$, and $(m_+n+k_n)/2=p_1n+\frac{\theta_n}2$.
Notice that $Y'_1(k_n)$ has mean $k_nr=p_1n+r\theta_n$. It again follows readily from Stirling that the  probability
in~(\ref{loc3}) is bounded from below by a constant times $1/\sqrt n$ (notice that if $p_1=0$, then $\theta_n$ also
vanishes). \qed

\section{Properties of the mean field free energy}
\label{sec:8s}

\setcounter{equation}{0}

To prove Proposition \ref{thm4s.1} we study the \emph{free energy} given by
    \eqref{4s.1} (two layers with a small vertical n.n. interaction).
We shall exploit the smallness of $\eps$, observing that for $\eps=0$ we have the
well known explicit expression:
    \begin{equation}
    \label{8s.2}
\hat f_0(m_1,m_2):= -\frac{m_1^2}2 -I(m_1) -\frac{m_2^2}2 - I(m_2),
    \end{equation}
where  the entropy $I(m)$ is given by
    \begin{equation}
    \label{8s.3}
I(m) = - \frac{1-m}2 \log \frac{1-m}2 - \frac{1+m}2 \log \frac{1+m}2,\;\; m \in [-1,1].
    \end{equation}
The function $-\frac 12 m^2 -I(m)$ is a symmetric  convex function
 of $m$  with a quartic minimum at 0 so that
    \begin{equation}
    \label{8s.4}
    \hat f_0(m_1,m_2) \ge \hat f_0(0,0) + c(m_1^4 + m_2^4).
    \end{equation}

\medskip

\begin{lem}
        \begin{equation}
    \label{8s.5}
 \hat f_0(m_1,m_2) - \eps \le   \hat f_\eps(m_1,m_2) \le \hat f_0(m_1,m_2) + \eps.
    \end{equation}
\end{lem}

\medskip

\begin{proof}
It follows at once from \eqref{2s.6} and \eqref{2s.1} since
$\hat\pi_0 (\cdot)-\eps \le \hat\pi_\eps(\cdot) \le \hat\pi_0 (\cdot)+\eps$.
%
\end{proof}

\medskip

\begin{cor}
There is $c'>0$ so that
    \begin{equation}
    \label{8s.7}
\inf_{m_1,m_2}\hat f_\eps(m_1,m_2) = \inf_{(m_1,m_2) \in \mathcal G_{c'}}\hat f_\eps(m_1,m_2)
    \end{equation}
 where
     \begin{equation}
    \label{8s.7a}
\mathcal G_c =\Big\{(m_1,m_2)\in [-1,1]\times  [-1,1]:|m_i| \le c  \eps^{1/4}, i=1,2\Big\}.
    \end{equation}

\end{cor}

\medskip

\begin{proof}
Using \eqref{8s.5} and \eqref{8s.4}, we easily see that
\[
  \hat f_\eps(m_1,m_2) \ge
   \hat f_\eps(0,0) + c(m_1^4 + m_2^4) - 2\eps,
   \]
so that $\hat f_\eps(m_1,m_2) \ge
   \hat f_\eps(0,0)$ if $(m_1,m_2) \notin \mathcal G_{c'}$ with $c'$ large enough,
   hence \eqref{8s.7}.
	\end{proof}

%
%

\vskip.5cm

We denote by $f_\eps(m)$, $m=(m_+,m_-)$, the  function $\hat f_\eps(m_1,m_2)$ when $m_1, m_2$ are
written in terms of $m_{\pm}$ as in \eqref{2s.4}. In the sequel  $m=(m_+,m_-)$, $h=(h_+,h_-)$ and
    \begin{equation}
    \label{3.18.1.1}
 \langle h, m \rangle:= h_+m_++h_-m_- = 2(m_1h_1+m_2h_2).
    \end{equation}
    Given $m$ and taking $h$ as in \eqref{2s.7} we then have
    \begin{equation}
    \label{3.18.3}
f_\eps(m)= -\frac{1}4 (m_+^2+m_-^2) + \phi_\eps(m) =
 -\frac{1}4 \langle m, m\rangle  + \frac 12 \langle h, m\rangle  -\pi_\eps(h).
    \end{equation}

%
%
%
%
By \eqref{8s.7} the inf of $f_\eps(m)$ is achieved
in the set $\mathcal G_c$ for $c$ large enough and the minimizers are critical
points   in  such a set. Denoting by $D$ the gradients, the critical points satisfy
    \begin{equation}
    \label{3.18.5}
Df_\eps(m) = -\frac{1}2 m + D\phi_\eps(m) =0.
    \end{equation}
Then by Lemma \ref{lemma5s.2} with $\theta=m$,
   \begin{eqnarray}
    \label{3.19}
m_+ &=& 2e^\eps \frac{\sinh (m_+ )}{e^\eps \cosh (m_+ )+e^{-\eps} \cosh (m_- )},\nonumber
\\
\\ m_- &=& 2e^{-\eps} \frac{\sinh (m_- )}{e^\eps \cosh (m_+ )+e^{-\eps} \cosh (m_- )}. \nonumber
    \end{eqnarray}
Of course $m_{+}=m_{-}=0$ is a solution.
%
In the next lemma we shall prove that
any   solution   has $m_-=0$.

\medskip

\begin{lem}
For any $x\in \mathbb R$ and any $\eps\ge 0$ the equation
   \begin{eqnarray}
    \label{3.19.0}
y &=& 2e^{-\eps} \frac{\sinh (y )}{e^\eps \cosh (x )+e^{-\eps} \cosh (y )},\quad y \in \mathbb R
    \end{eqnarray}
has a unique solution: $y=0$.

\end{lem}

\begin{proof} If $y$ solves \eqref{3.19.0} then so does $-y$. Therefore we only need to prove that there is no solution with $y>0$.  Define
    \[
  U(y):=\frac{ 2e^{-\eps} }{e^\eps +e^{-\eps} \cosh (y )} \sinh (y ).
    \]
Since $U(y)$ is not smaller than the right hand side of  \eqref{3.19.0},
the lemma will be proved once we show that $U(y)<y$ for all $y>0$.  Suppose by contradiction that
there is $y>0$ so that $y \le U(y)$.  Then
    \[
{e^\eps +e^{-\eps} \cosh (y )}  \le{ 2e^{-\eps} } \frac{\sinh (y )} y,
    \]
which yields:
   \[
\{e^\eps +e^{-\eps}\} + e^{-\eps}\sum_{n\ge 1}  \frac{y^{2n}}{(2n)!}  \le{ 2e^{-\eps} }
+e^{-\eps} \sum_{n\ge 1}  \frac{2y^{2n}}{(2n+1)!}.
    \]
But this last inequality is not true, since $ e^\eps +e^{-\eps} \ge  2e^{-\eps}$ and
$(2n+1)!> 2(2n)!$ for $n\ge 1$.
\end{proof}


%
%
%
%
%
%

\medskip
We can thus put $m_-=0$ in the first equation of
\eqref{3.19} which then becomes
an equation for $m_+$ alone. The proof of Proposition \ref{thm4s.1} is then a consequence of
the following lemma:

\medskip

\begin{lem}
There is $\delta>0$ so that for all $\eps>0$ small enough the equation
   \begin{eqnarray}
    \label{3.19.1}
x &=& 2e^{\eps} \frac{\sinh (x )}{e^\eps \cosh (x )+e^{-\eps}},\quad x \in (0,\delta)
    \end{eqnarray}
has a unique solution $x_\eps$ and
   \begin{eqnarray}
    \label{3.19.2}
|{x_\eps}- \sqrt{ {12}{ \eps}}| \le c \eps^{3/2}.
    \end{eqnarray}

\end{lem}

\noindent
\begin{proof}  \eqref{3.19.1} can be rewritten 
   \begin{eqnarray}
    \label{3.19.3}
e^\eps \cosh (x )+e^{-\eps} &=& 2e^{\eps} \frac{\sinh (x )}x,\quad x \in (0,\delta)
    \end{eqnarray}
and therefore as $F(x^2)=0$, where for $z \ge 0$
   \begin{eqnarray}
    \label{3.19.4}
F(z) &=& \{e^\eps+e^{-\eps} + \frac{z}{2}\; e^\eps+ z g_1(z) \}-\{2e^{\eps} + \frac{z}{3}\;e^\eps + z g_2(z)\}
    \end{eqnarray}
and, 
   \begin{eqnarray*}
g_1(z) = e^{\eps}\sum_{n\ge 2}  \frac{z^{n-1}}{(2n)!},\quad
g_2(z) = 2e^{\eps}\sum_{n\ge 2}  \frac{z^{n-1}}{(2n+1)!}.
    \end{eqnarray*}
We have:
   \begin{eqnarray}
    \label{3.19.5}
F(12\eps) &=& \{e^\eps+e^{-\eps} + 6\eps \; e^\eps+ (12\eps) g_1(12\eps) \}-\{2e^{\eps} + 4\eps \;e^\eps + (12\eps) g_2(12\eps)\}\nonumber\\
&=& c \eps^2.
    \end{eqnarray}
On the other hand for any $a< 1/6$ there are $\delta_0$ and $\eps_0$ positive so that
for any  $\delta\in (0,\delta_0)$ and $\eps\in (0,\eps_0)$ we have
\[
\frac{dF}{dz} \ge a \text{ on }  (0,\delta),
\]
Hence there is a unique $z^*\in (0,\delta)$ so that $F(z^*)=0$. Moreover
\[
|z^* - {12}{\eps}| \le \frac{c\eps^2}{a},
\]
proving the lemma.
\end{proof}
%
%
%
%

\medskip


    \medskip

\begin{thm}
\label{thm4.2}
For any $\eps>0$ small enough the Hessian $D^2f_\eps$ of the free energy $f_\eps$
is positive definite at the  minimizers
$\pm m^{(\eps)}$.

\end{thm}

\medskip

\noindent
\begin{proof} Since we shall be dealing with functions in the $m$ and $h$ domain, to prevent confusion we write $D_m$ and $D_h$
for the corresponding gradients (and for the Jacobian matrices). Analogously for the corresponding Hessian matrices. From \eqref{3.18.5} and \eqref{2s.6}, by differentiating we have:
 $\dis{D^2_m f_\eps= - \frac 12 \mathbb{I} + D^2_m \phi_\eps}$ and
\[
D_m \phi_\eps = \frac h2,\quad D^2_m \phi_\eps = \frac {D_mh}2.
\]
Since $h(m)$ is the inverse of $m(h)$:
\[
D_h m D_m h  = \mathbb{I}.
\]
On the other hand, by \eqref{2s.8.2},  $m = 2 D_h \pi$, so that 
\[
D_h m = 2G,\quad G = D^2_h\pi
\]
and therefore
\[
2G D_m h  = \mathbb{I},\quad D_m h = \frac 12  G^{-1}
\]
and, in conclusion,
\[
D^2_m \phi_\eps =  \frac 14  G^{-1}
\]
Thus
\[
  D^2_m f_\eps=  - \frac 12\mathbb{I}  + \frac{1}{4} G^{-1} =  \frac 1{4}G^{-1} \Big( \mathbb{I}-2G\Big).
\]
The elements of $G$ are given in the equations which follow
\eqref{2s.3} and they must be computed at $m^{(\eps)}$,
so that $m_-=0$. Then $G$ is diagonal and its diagonal elements, denoted by $G_{++}$
and $G_{--}$, are:
\begin{eqnarray}
\label{9s.1}
2G_{++}   &=&  2e^\eps\frac{\cosh (h_+ )[e^\eps\cosh (h_+ )+ e^{-\eps}] -e^{\eps} \sinh^2 (h_+ )}
{[e^\eps\cosh (h_+ )+ e^{-\eps}]^2}\nn\\
&=& 2e^\eps\frac{e^\eps   + e^{-\eps}\cosh (h_+ )}{[e^\eps\cosh (h_+ )+ e^{-\eps}]^2} \le
\frac{2e^\eps}{e^\eps\cosh (h_+ )+ e^{-\eps}} \le 1 - \frac{\eps}{2}.
\end{eqnarray}
The last inequality holds for $\eps$ small enough and it is proved as follows.
We develop in Taylor series all terms up to first order in $\eps$, thus the equality below
are meant modulo terms in $\eps^2$. Recalling \eqref{3.19.2} we also bound from below $h_+^2=m_\eps^2 > 8\eps$.
The last fraction in \eqref{9s.1} is then bounded by
\[
\frac{2+2\eps}{(1+\eps)(1+h_+^2/2)+1-\eps}<\frac{2+2\eps}{2 + 4\eps}
=\frac{1+ \eps}{1 + 2\eps} \le 1 -  \eps,
\]
hence \eqref{9s.1} for $\eps$ small enough.  We have
\begin{eqnarray}
\label{9s.2}
2G_{--}   &=&
\frac{2e^{-\eps}}{e^\eps\cosh (h_+ )+ e^{-\eps}} \le  1 - \frac{\eps}{2}
\end{eqnarray}
for $\eps$ small enough (and for any value of $h_+$).
We have thus seen
that $\mathbb{I}-2G$ is diagonal and its diagonal elements are $\ge \frac\eps 2$
for $\eps$ small.
\end{proof}

    \medskip

\begin{cor}
\label{coro4.1}
For any $\eps>0$ small enough there is $c>0$ so that for any $\zeta>0$ small enough:
    \begin{equation}
    \label{4.19.0}
    \Big| f_\eps(m)- f_\eps(m^{(\eps)})\Big|  \ge c \zeta^2,\quad \text{for all $m$ such that
    $|m\mp m^{(\eps)}| \ge \zeta$}.
    \end{equation}

\end{cor}

\medskip

\noindent
\begin{proof} From what was already seen, the inequality holds if $|m| \ge
c \eps^{1/4}$ with $c$
large enough.  The infimum of $f_\eps(m)$ in  $|m\mp m^{(\eps)}| \ge \zeta$ must then be achieved
in the set
    \[
    \{|m\mp m^{(\eps)}| \ge \zeta\} \cap \{|m| \le
c \eps^{1/4}\}.
    \]
In such a set $D_m f_\eps \ne 0$, thus the infimum must
be achieved at the boundaries, hence \eqref{4.19.0}.

\end{proof}

\setcounter{equation}{0}

\section{Multi-canonical constraints}
\label{sec:7s}
%

The setup is the following:
$I=C^{\ell_-}_0=[0,\ell_-)\cap \Z$ where, recalling \eqref{2.2}, $\ell_- = \ga^{-(1-\alpha)}$, $\alpha>0$ and small. Let
$(m_1(x), m_2(x)) \in [-1,1]^2$ for $x \in I$, $(\bar m_1(x), \bar m_2(x))\in [-1,1]^2$ for $x \in \Z \setminus I$.
Dropping the dependence on $\gamma$ and $I$, let
    \begin{eqnarray}
    \label{5.1}
 \mathcal F (m\,|\,\bar m)& = &\sum _{x\in I} \hat\phi_\eps (m_1(x),m_2(x)) - \sum_{i=1,2}\Big\{\frac 12 \sum_{x\ne y\in I} J_\ga(x,y)
 m_i(x)m_i(y)\nonumber \\ &+& \sum_{x\in I, y\notin I}  J_\ga(x,y)m_i(x)\bar m_i(y)\Big\}
    \end{eqnarray}
where $\hat\phi_\eps$ is the canonical free energy in Proposition \ref{thm3s.2}.

Proposition \ref{thm3s.4} follows at once from the result below, which is
the analogue for two layers
of Theorem 6.4.1.1 of \cite{presutti}, after the
vertical interaction is added in.

\medskip

\begin{prop}
\label{multi}  
For all $\ga$ small enough and all
$u=(u_1,u_2)$, $u_i\in [-1,1]$, there is a unique
$\hat m$ such that
$\dis{\sum_{x\in I}  \hat m_i(x)= |I|\,u_i}$, $i=1,2$, and
    \begin{eqnarray}
    \label{5.3}
 \mathcal F (m\,|\,\bar m)\ge \mathcal F (\hat m\,|\,\bar m),\quad \text{for all $m$ such that }\; \sum_{x\in I} m_i(x)= |I|\,u_i,\;i=1,2.
    \end{eqnarray}
Moreover the minimizer $\hat m$ is smooth in the sense that
there is a constant $c$ so that 
\[
\max_{i=1,2}\max_{x \in I}|\hat m_i(x)-u_i| \le c \ga^\alpha.
\]

\end{prop}

\medskip

\noindent
{\it  Proof.} The statement is trivially true when $|u_1|\vee |u_2|=1$,
We therefore assume in the sequel that $|u_1|\vee |u_2|< 1$.

\noindent \emph{Remark.} As in Section \ref{sec:5s}, for the free energy computation it is sometimes convenient to use the
variables $m_{\pm}(x), x \in I$ as in \eqref{2s.4}; we then write  $m(x)=(m_+(x), m_-(x))$ and analogously
$\bar m(x)= (\bar m_+(x),\bar m_-(x))$, and write  $\hat\phi_\eps(m_1(x), m_2(x))=\phi_\eps(m(x))$,
with the function $\phi_\eps$ given by \eqref{2s.7}.

To find the minimizers under the above constraint, we introduce the Lagrange multipliers
$\la=(\la_+,\la_-)$
and define
    \begin{eqnarray}
    \label{5.4}
 \mathcal F_\la (m\,|\,\bar m)& = &
\sum _{x\in I}\Big( \phi_\eps (m(x)) - \frac 12 \langle m(x), \la\rangle  - \frac 12 \langle \kappa(x),  m(x)\rangle \Big)
\nonumber \\ &-& \frac 14  \sum_{x\ne y\in I} J_\ga(x,y)
\langle m(x), m (y)\rangle  + \frac 12 \langle \la, u\rangle  |I|
    \end{eqnarray}
with $u_\pm= u_2\pm u_1$, $\langle a, b\rangle = a_+b_++a_-b_-$ and
    \begin{eqnarray}
    \label{5.5}
\kappa_\pm(x) := \sum_{y\notin I} J_\ga(x,y) \bar m_{\pm}(y),\quad x\in I.
     \end{eqnarray}
Observe that
$ \mathcal F_\la (m\,|\,\bar m)= \mathcal F  (m\,|\,\bar m)$
for all $m$ under the constraint in \eqref{5.3}.
Let $\bar\kappa=\frac {1}{|I|}\sum_{x\in I}\kappa(x)$. We
introduce an interpolating parameter  $s\in [0,1]$ and  define
    \begin{eqnarray}
    \label{5.6}
 \mathcal F_{\la,s} (m\,|\,\bar m)& = &\sum_{x\in I} \{\phi_\eps (m(x))- \frac 12 \langle m(x), [\la+\bar\kappa]\rangle
    -s\Big( \sum_{x\in I}\{\frac 12 \langle [\kappa(x)-\bar\kappa],
 m(x)\rangle \}\nonumber \\ &-&
 \frac 14 \sum_{x\ne y \in I} J_\ga(x,y)
\langle m(x), m (y)\rangle  \Big) + \frac 12 \langle \la, u\rangle  |I|
\end{eqnarray}
so that $ \mathcal F_{\la,1}= \mathcal F_{\la}$. To find the minimizer of $ \mathcal F_{\la,s} (m\,|\,\bar m)$
we need to find its critical points, namely the solutions of
   \begin{equation}
    \label{5.8}
\frac{\partial \mathcal F_{\la,s} (m\,|\,\bar m)}{\partial m(x)}=D\phi_\eps( m(x)) - \frac 12 \theta(x) = 0
\end{equation}
where
     \begin{eqnarray}
    \label{5.8.00}
\theta(x)= \frac 12 (\la +\bar\kappa)  + \frac s2  \big(\kappa(x)-\bar\kappa \big) + \frac s2 \sum_{y\ne x, y\in I} J_\ga(x,y) m(y)
   \end{eqnarray}
By an explicit computation:
   \begin{equation*}
\frac{\partial^2 \mathcal F_{\la,s} (m\,|\,\bar m)}{\partial m(x)\partial m(y)}=
D^2 \phi_\eps (m(x))\mathbf 1_{x=y} - sJ_\ga(x,y) \mathbf 1_{x\neq y\in I}
\end{equation*}
which   for $\ga$ small enough  is a positive  symmetric operator, namely there is
$c>0$ so that
    \begin{eqnarray}
    \label{5.7}
 \sum_{x\in I} \langle\psi(x),  D^2 \phi_\eps (m(x))\psi(x)\rangle -
\sum_{x\ne y \in I} sJ_\ga(x,y) \langle\psi(x), \psi(y)\rangle
 \ge        c      \sum_{x\in I} \langle\psi(x), \psi(x) \rangle
   \end{eqnarray}
This shows that if there is a critical point  of $\mathcal F_{\la,s}$ it is unique and it
 minimizes  $\mathcal F_{\la,s}$.
By  \eqref{5.8}  $m$ is a critical point if $D  \phi_\eps (m(x)) = \frac {\theta(x)}2$
for all $x$.  By Lemma \ref{lemma5s.2}, $m(x) = 2 D \pi_\eps (\theta(x))$
which is \eqref{2s.8.2} namely
   \begin{eqnarray}
    \label{5.11}
m_+(x) &=& 2e^\eps \frac{\sinh (\theta_+(x) )}{e^\eps \cosh (\theta_+(x) )+e^{-\eps} \cosh (\theta_-(x) )} \nonumber
\\
\\ m_-(x) &=& 2e^{-\eps} \frac{\sinh (\theta_-(x) )}{e^\eps \cosh (\theta_+(x) )+e^{-\eps} \cosh (\theta_-(x) )}. \nonumber
    \end{eqnarray}
Observing that the right hand side depends weakly on $m$ as $|\sum_{y\in I} J_\ga(x,y) m(y)| \le \ga^{\alpha}$, we
then get, as in the proof of  Theorem 6.4.1.1 in \cite{presutti}, that
there is a unique solution $m_{\la,s}$ of \eqref{5.11} which is the unique
minimizer of $\mathcal F_{\la,s}$ and whose fluctuations are of order $O(\ga^{\alpha})$.
To conclude the proof of the theorem it suffices to show that there exists a function $\la(s)$ such that
   \begin{equation}
    \label{5.12}
  \sum_{x\in I} m_{\la(s),s}(x) = u\, |I|,\quad \text{for all $s\in [0,1]$}.
    \end{equation}
There is obviously
a unique solution $\la(0)$ of \eqref{5.12} when $s=0$.  We will find $\la(s)$ by solving the
evolution equation
   \begin{equation}
    \label{5.10.0}
\sum_{x\in I}  \frac{d}{ds}
\Psi (\theta_{\la(s),s}(x))=\sum_{x\in I}
D_\theta\Psi (\theta_{\la(s),s}(x))\frac{d}{ds} \theta_{\la(s),s}(x)= 0,\quad s\in [0,1]
    \end{equation}
obtained by differentiating  \eqref{5.12} and recalling
that  $m= 2 D \pi_\eps (\theta)=:2\Psi(\theta)$ where the explicit expression of $2\Psi=2(\Psi_+(\theta),\Psi_-(\theta)) $
is given by the r.h.s. of \eqref{5.11}.

We will now prove that $\la(s)$ is differentiable and its derivative has order $\ga^\alpha$.
We proceed by supposing that $\la$ is differentiable and get a formula for its derivative.
We will then check that the primitive for such an expression is indeed  $\la(s)$.  The formula
will also show that  $\la'(s)$  has order $\ga^\alpha$.

By \eqref{5.10.0} we first need to show that $\theta_{\la(s),s}$ is differentiable in $s$.
But $\theta_{\la(s),s}(x)$ defined through \eqref{5.8.00} can be expressed as
\begin{equation}\label{th01}
  \theta_{\la(s),s}(x)=\theta^{(0)}_{\la(s)}+s\;\theta^{(1)}_{\la(s)}(x),
\end{equation}
where $\theta^{(0)}_{\la(s)}=\frac 12 (\la+\bar \kappa)$ and
$\theta^{(1)}_{\la(s)}(x)=\frac 12  \bigg[\big(\kappa(x)-\bar\kappa \big) +  \sum_{y\ne x, y\in I} J_\ga(x,y) m_{\la(s)}(y)\bigg]$, so that
\begin{equation}\label{th1bound}
  |\theta^{(1)}_{\la(s)}(x)|\le c\ga^\alpha.
\end{equation}
We have:

\begin{eqnarray}
  \frac{d\theta_{\la(s),s}}{ds}(x) &=& {\partial}_\la \theta_{\la(s),s}(x)\frac{d\la}{ds}+
	\frac{\partial\theta_{\la(s),s}}{\partial s}(x)
\\
\nn && \hskip -5cm\text{with:}
\\
\label{pstheta}
\frac{\partial\theta_{\la(s),s}}{\partial s}(x)&=&\theta^{(1)}_{\la(s)}(x)
\\
\label{pltheta}
{\partial}_\la \theta_{\la(s),s}(x)&=&{\partial}_\la \theta^{(0)}+ s\;{\partial}_\la \theta^{(1)}=
\frac 12 \mathbb{I}+s\; \sum_{y\in I}J_\ga(x,y)D_\la m_\la,
\end{eqnarray}
where ${\partial}_\la := \bigg(\frac{\partial }{\partial \la_+},\frac{\partial }{\partial \la_-}\bigg)$, $\mathbb{I}$ is the identity matrix, and $m_\la$ is defined by equation:
\begin{eqnarray}
\label{mtheta}
m(x)= \Psi(\tilde\theta_{\la, m}(x))
\end{eqnarray}
 where $\tilde\theta_{\la, m}$ as the same expression of $\theta$ \eqref{5.8.00} but as a function of $\la, m$, namely:
\begin{equation}\label{tildetheta}
  \tilde\theta_{\la, m}(x)=\frac 12 (\la+ \bar\kappa)-\frac s2\bigg[(\kappa(x)-\bar\kappa)+\sum_{y\neq x,y\in I} J_\ga(x,y) m(y)\bigg].
\end{equation}

We postpone the proof that $m_\la$ is differentiable in $\la$ and there is a constant $c$ so that:
\begin{equation}\label{Dmla}
  |D_\la m|<c \ga^{\alpha}
\end{equation}
which, recalling \eqref{pltheta},  implies that there is a constant $c_J$ so that:
\begin{equation}\label{Dmlb}
  \frac 12 - c_J\ga^\alpha<\|\partial_\la \theta\|< \frac 12 + c_J\ga^\alpha.
\end{equation}
Going back to \eqref{5.10.0} and
denoting by $\Upsilon_\theta$ the operator that acts on a vector $v(x)$ as $\Upsilon_\theta v(x):=\sum_{x\in I}D_\theta\Psi(\theta_{\la(s),s}(x)) v(x)$, we write in a compact form:
\begin{eqnarray}
[\Upsilon_\theta{\partial}_\la \theta_{\la(s),s}] \frac{d\la}{ds}&=& -  \Upsilon_\theta \theta^{(1)}_{\la(s),s}.
\end{eqnarray}

There are positive constants $c^\pm_\eps$ (see \eqref{5.11}) so that:
\begin{equation}
|I|c^-_\eps<\|\Upsilon_\theta\|<|I|c^+_\eps.
\label{12a}
\end{equation}
Then, by \eqref{Dmlb}, $[\Upsilon_\theta{\partial}_\la \theta_{\la(s),s}] $ is invertible if $\ga$ is small enough.  Hence $\frac {d\la}{ds}$ is well defined and
\begin{equation}\label{bound}
  \left| \frac{d\la}{ds}\right|= \bigg|\Big[\Upsilon_\theta{\partial}_\la \theta_{\la(s),s}\Big]^{-1}\Upsilon_\theta \theta^{(1)}_{\la(s),s}\bigg|<c\ga^\alpha
\end{equation}
Last inequality follows by \eqref{pstheta}.
\vskip 1cm
Let us now prove \eqref{Dmla}. Differentiating \eqref{mtheta}:
\begin{eqnarray}
D_\la m= D_{\tilde\theta} \Psi \bigg[\partial_\la \tilde \theta+
\partial_m \tilde\theta \cdot D_\la m\bigg]
\end{eqnarray}
then
\begin{eqnarray}
\bigg[1-D_{\tilde\theta} \Psi \cdot \partial_m \tilde\theta\bigg]D_\la m= D_\theta \Psi \cdot
\partial_\la \tilde\theta
\end{eqnarray}
 By \eqref{tildetheta} we see that $\partial_\la \tilde\theta=1/2 \mathbb I$,
 and $ D_m \tilde\theta$ is  a smooth function order $o(s\ga^{\alpha})$.
Then for $\ga$ small enough $\|D_\theta \Psi D_m \theta\|<1$ and
$\bigg[1-D_\theta \Psi D_m \theta\bigg]$ is invertible and
\begin{eqnarray}
\label{04}
D_\la m=\bigg[1-D_\theta \Psi D_m \tilde\theta\bigg]^{-1} D_\theta \Psi \cdot \partial_\la \tilde\theta
\end{eqnarray}
is well defined and there is a constant $c$ so that:
\begin{equation}\label{Dm}
  |D_\la m|< c  \ga^\alpha.
\end{equation}


\vskip1cm

\setcounter{equation}{0}

\section{A contraction property of the mean field free energy}
\label{sec:9s}

In this section we shall prove Proposition \ref{thm4s.3.1}.
The basic bound comes from the analysis of the previous section.
As we use here the variables $(m_1,m_2)$ instead of $(m_+,m_-)$
we need to translate the results into the new variables.
The minimizer $m^{(\eps)} = (m^{(\eps)}_+,0)$ becomes
$(m_\eps,m_\eps)$, $m_\eps = m^{(\eps)}_+/2$.  We have

    \medskip

\begin{lem}
\label{coro4.0}
Let $G_{++}$ and $G_{--}$ be given by \eqref{9s.1}--\eqref{9s.2}, then
    \begin{equation}
    \label{9s.3}
   \frac {\partial^2 \hat \pi_\eps}{\partial h^2_i}(m_\eps,m_\eps)  = G_{++}+G_{--},\; i=1,2,\quad
   \frac {\partial^2 \hat \pi_\eps}{\partial h_2\partial h_1}(m_\eps,m_\eps)  = G_{++}-G_{--}.
   \end{equation}
Both $G_{++}+G_{--}$ and $G_{++}-G_{--}$ are non negative, and for $i=1,2$
    \begin{equation}
    \label{9s.3a}
  |\frac {\partial^2 \hat \pi_\eps}{\partial h^2_i}(m_\eps,m_\eps)| + |  \frac {\partial^2 \hat \pi_\eps}{\partial h_2\partial h_1}(m_\eps,m_\eps)|
  = 2 G_{++}\le 1-\frac \eps 2.
   \end{equation}

\end{lem}

    \medskip

\begin{proof}
    \begin{equation*}
   \frac {\partial  \hat \pi_\eps}{\partial h _1}   = \frac {\partial   \pi_\eps}{\partial h _+}
   -\frac {\partial   \pi_\eps}{\partial h _-},\quad
   \frac {\partial  \hat \pi_\eps}{ \partial h_2}  = \frac {\partial   \pi_\eps}{\partial h _+}
   +\frac {\partial   \pi_\eps}{\partial h _-}
   \end{equation*}
       \begin{equation*}
   \frac {\partial^2  \hat \pi_\eps}{\partial h _1^2}   = \frac {\partial^2   \pi_\eps}{\partial h _+^2}
   +\frac {\partial^2   \pi_\eps}{\partial h _-^2} - 2
   \frac {\partial^2  \pi_\eps}{\partial h _+\partial h_-},\quad
   \frac {\partial^2  \hat \pi_\eps}{ \partial h_2\partial h _1}  = \frac {\partial^2   \pi_\eps}{\partial h _+}
   -\frac {\partial^2   \pi_\eps}{\partial h _-^2},
   \end{equation*}
hence \eqref{9s.3} because
$\frac {\partial^2   \pi_\eps}{\partial h_+ \partial h_-}(2m_\eps,0)=0$.

\end{proof}

It follows by continuity that:

\begin{cor}
\label{cor9s.1}
There are  $c_0>0$ and $r<1$ so that
the following holds. Call
    \begin{equation}
    \label{9s.4.1}
R_{i,j} = \sup_{h=(h_1,h_2):|h_i-m_\eps| \le c_0, i=1,2}  |  \frac {\partial^2 \hat \pi_\eps}{\partial h_i\partial h_j}(h)|. 
    \end{equation}
Then
    \begin{equation}
    \label{9s.4}
\sum_{j=1,2} R_{i,j} 
 \le r, \quad i=1,2
    \end{equation}
and the matrix $(1-R)$ is invertible.
\end{cor}

\medskip

We are now ready for the proof of Proposition \ref{thm4s.3.1}.  We
fix throughout the sequel a pair $(x,i)$ and $(x,i')$ of vertically interacting sites in $\Delta_{\rm in}$,
thus $(x,i')= v_{x,i}$ and  use the above properties to study the  function $g_\eps(m)$
introduced in \eqref{9s.5}.
We write $m=(m_i,m_{i'})$ and write  $a_i=a_{x,i}$, $a_{i'}
=a_{x,i'}$, the latter defined in \eqref{4s.5.23}.  We also write
$\la_i$ and $\la_{i'}$ dropping the superscript $u$ on which they depend via \eqref{9s.7}.
We finally shorthand $a_j$ for $a_{x,j}$.

\begin{lem}
There is a constant $c$ so that for any $u\in \mathcal N_{x,i,i'}$, see \eqref{9s.8},
   \begin{eqnarray}
    \label{9s.9}
&&
\frac{| \la_j -\la_j^{\rm eq}|}{1-a_j} \le\zeta + c \ga^\alpha,\quad j=i,i',
    \end{eqnarray}
where $\la_j^{\rm eq}$ is defined in Proposition \ref{thm4s.3.1}.

\end{lem}

\begin{proof}
By \eqref{9s.7}
   \begin{eqnarray*}
\la_j -\la_j^{\rm eq} &=&   \sum_{y\ne x:(y,j)\in C^{\ell_-,j}_x}
J_\ga(x,y) \big(u(y,j)-m_\eps)  \\&+&
 \sum_{y:(y,j)\notin \{C^{\ell_-,j}_x \cup \Delta_0\}
}
J_\ga(x,y) \big(u(y,j)-m_\eps).
    \end{eqnarray*}
We
add and subtract $\hat J_\ga(x,y)$ to $J_\ga(x,y)$ where $\hat J_\ga(x,y)$
is obtained by averaging  $J_\ga(x,y')$ over $ C^{\ell_-,j}_y$.  In the term with
$\hat J_\ga(x,y)$ we can replace $u(y,j)$ by its average and use \eqref{9s.8},
\eqref{4s.5.23} and that $J_\ga$ is a probability kernel to get the bound $(1-a_j) \zeta$.
The sum over the terms with $\hat J_\ga(x,y)-J_\ga(x,y)$ is bounded by $c'\ga^\alpha$,
by the smoothness of $J_\ga$.

\end{proof}

In the sequel we shall only use the bound \eqref{9s.9} and not the specific form of the $\la_j$.

By differentiating \eqref{9s.5} we get
   \begin{equation*}
\frac {\partial^2  g_\eps}{\partial m_j \partial m_{j'}}  =
   \frac {\partial^2  \hat \phi_{\eps}}{\partial m_j \partial m_{j'}} -a_j\mathbf 1_{j=j'},\quad j,j' \in \{i,i'\}.
      \end{equation*}
At $\eps=0$ $\frac {\partial^2  g_0}{\partial m_j \partial m_{j'}} $ is diagonal
with entries  $-I''(m_j) - a_j$, $j=i,i'$.  The minimum of $-I''$ is at 0
and $-I''(0)=1$.  Since $a_j \le 1/2$ (this follows from the choice of $x$ and the symmetry of $J$; see~(\ref{4s.5.23})),
we then conclude that:

\begin{lem}
There is $c_1>0$ so that $g_\eps$ is strictly convex  for $\eps\le c_1$
and for any such $\eps$ it has a unique minimizer $\tilde m$ (called
$m^{(u)}$ in   Proposition \ref{thm4s.3.1}).

\end{lem}
 In the sequel we tacitly suppose $\eps\le c_1$. The critical point $m$ of $g_\eps$ satisfies
   \begin{equation}
    \label{9s.10}
\hat D   \hat\phi_\eps(m)  = \hat \theta  :=(a_i m_i +\la_i, a_{i'} m_{i'} +\la_{i'})
    \end{equation}
Then, by Lemma \ref{lemma5s.2bis},
   \begin{equation}
    \label{9s.11}
m  = \hat D_h \hat \pi_\eps (\hat\theta ),
    \end{equation}
where $\hat D_h\hat \pi_\eps(\hat\theta ) $ is the gradient of $\hat \pi_\eps(h_1,h_2)$
computed  at $(h_1,h_2) = \hat\theta$.  We shall study \eqref{9s.11} distinguishing
among the possible values of   $a_i$ and $a_{i'}$ which
depend on the horizontal distance of $(x,i)$ and respectively $(x,i')$ from $\Delta_0$.
Because of the geometric properties  of $\Delta_{\rm in}$,  only three cases can occur: (i)  $a_i=a_{i'} =0$; (ii) $a_i=a_{i'} \in (0, \frac 12]$; (iii) $a_i\in (0, \frac 12]$
and $a_{i'}=0$ or viceversa.
Case (i) occurs when the horizontal distances of
$(x,i)$ and  $(x,i')$ from $\Delta_0$ are both $>\ga^{-1}$. Case (ii) is when the distances of
$(x,i)$ and  $(x,i')$ from $\Delta_0$ are both $\le \ga^{-1}$ and case(iii) is when one is $\le \ga^{-1}$ and the other
$> \ga^{-1}$. We start from case (i) which is the easiest.

\medskip

Case (i).
By \eqref{9s.9} for  $\ga$ small enough $|\hat\theta_j-m_\eps| < c_0$, $c_0$ as in
Corollary \ref{cor9s.1}, hence by \eqref{9s.9}
   \begin{eqnarray*}
&&|m_j-m_\eps| \le \sum_{j'}R_{j,j'}|\la_{j'} -\la_{j'}^{\rm eq}|, \quad
\sum_{j'}R_{j,j'}\le r <1
    \end{eqnarray*}
in agreement with  in   Proposition \ref{thm4s.3.1} after setting
$C_{x,i,i'}(j,j')=R_{j,j'}$ and recalling that in case (i) $a_j=0$, $j=i,i'$.

    \medskip

Case (ii). $\hat\theta$ in
\eqref{9s.11}  is now (after adding and subtracting $m_\eps$)
   \begin{equation}
    \label{9s.13}
 \hat \theta_j= m_\eps + a(m_j-m_\eps) +\la_j -(1-a)m_\eps,\quad j=i,i'.
    \end{equation}
Since $\hat \theta$ depends on $m$, \eqref{9s.11} is an equation in $m$ and not
a formula for $m$ as in case (i).
We introduce an interpolating parameter $t\in [0,1]$ and define
       \begin{equation}
    \label{9s.15}
 \hat \theta_j(t)= m_\eps + a(m_j-m_\eps) +t\Big(\la_j -(1-a)m_\eps\Big),\quad j=1,2
    \end{equation}
calling $m(t)$ the solution of  \eqref{9s.11} with  $\hat \theta$ replaced by $\hat \theta(t)$.
Observe that $m(0)=m_\eps$ is the solution at $t=0$ while the solution at $t=1$
is what we want to find because   $ \hat \theta(1)= \hat \theta$.

Supposing that $m(t)$ and its derivative $\dot m(t)$  exist we can then
differentiate  \eqref{9s.11} to get
       \begin{equation}
    \label{9s.16}
\dot m_j =\sum _{p=i,i'}K_{jp}\{ a \dot m_p + (\la_p -(1-a)m_\eps)\},\quad
K_{jp} =
\frac {\partial^2 \hat \pi_\eps}{\partial h_j\partial h_p}(\hat \theta(t))
    \end{equation}
where $\hat \theta(t)$ is computed at $m=m(t)$.  If moreover $|m(t)-m_\eps| \le 2 \zeta$
then $|\hat \theta_j(t)- m_\eps|\le 2 \zeta$ and $1- aK$ is invertible and we have
       \begin{equation}
    \label{9s.17}
\dot m  = V(m,t):=(1-aK)^{-1} K(\la  -(1-a)u^{\rm eq}),\quad u^{\rm eq}=(m_\eps,m_\eps),\; \la=(\la_i,\la_{i'}).
    \end{equation}
The evolution equation \eqref{9s.17} starting from $m(0)=m_\eps$ has a unique solution
till the first time $T$  when $|m_j(T) - m_\eps| = c_0$,  because by Corollary \ref{cor9s.1}, $(1-K(t))$ is invertible
and smooth for $t\le T$ and we have
      \begin{equation}
    \label{9s.18}
|m_j(t)-m_\eps| \le t \sum_{n=0}^\infty a^n\sum_{p=i,i'} (R^{n+1})_{jp} | \la_p- (1-a)m_\eps|.
    \end{equation}
 Set
 \[
 C_{x,i,i'}(j,j')   = (1-a) \sum_{n=0}^\infty a^n (R^{n+1})_{j,j'}.
  \]
 Then, by \eqref{9s.4}, we get
  \[
  \sum_{j'=i,i'} C_{x,i,i'}(j,j') \le (1-a)\frac{r}{1-ar} < r , \quad j=i,i'
  \]
 Thus  $|m_j(t)-m_\eps| \le 2\zeta$ for $t \le \min\{T,1\}$,
 hence the above holds till $t=1$ and Proposition \ref{thm4s.3.1} is proved in case (ii).

    \medskip

Case (iii) with $a_i=a>0$ and $a_{i'}= 0$ (same proof applies when $a_i=0$ and $a_{i'}>0$).  Here
   \begin{equation*}
 \hat \theta = \Big(m_\eps + a(m_i-m_\eps) +\la_i -(1-a)m_\eps, m_\eps +(\la_{i'}-m_\eps)\Big)
    \end{equation*}
 and proceeding as in case (ii) we set
    \begin{equation}
    \label{9s.21}
 \hat \theta (t) = \Big(m_\eps + a(m_i-m_\eps) +t[\la_i -(1-a)m_\eps], m_\eps +t(\la_{i'}-m_\eps)\Big).
    \end{equation}
Analogously to \eqref{9s.16},
       \begin{equation*}
\dot m_i = K_{i,i} a \dot m_i +\{ K_{i,i}(\la_i - (1-a)m_\eps) + K_{i,i'}(\la_{i'} - m_\eps) \}.
    \end{equation*}
Since $ K_{i,i} a<1$ till when $|m_i(t)-m_\eps| < 2\zeta$ proceeding as in case (ii) we get
that the evolution equation has solution till time $t=1$ and
       \begin{equation}
    \label{9s.22}
|m_i(1)-m_\eps| \le   \sum_{n=0}^\infty (a R_{i,i})^n   \{ R_{i,i}|\la_i - (1-a)m_\eps| + R_{i,i'}|\la_{i'} - m_\eps| \}.
    \end{equation}
We then set:
       \begin{equation}
    \label{9s.23}
 C_{x,i,i'}(i,i)= \frac{R_{i,i}(1-a)}{1-aR_{i,i}},\quad
 C_{x,i,i'}(i,i')= \frac{R_{i,i'}}{1-aR_{i,i}}
    \end{equation}
which verifies the condition in Proposition \ref{thm4s.3.1} because
       \begin{equation*}
\sum_{j=i,i'}C_{x,i,i'}(i,j)= \frac{R_{i,i}(1-a)}{1-aR_{i,i}}+
  \frac{R_{i,i'}}{1-aR_{i,i}}\le \frac{r-aR_{i,i}}{1-aR_{i,i}} < r.
    \end{equation*}
Since   $\dis{m_{i'}  = \frac{\partial\hat \pi_\eps}{\partial h_{i'}} (\hat \theta  )}$,
    \[
|m_{i'} -m_\eps| \le R_{i',i'} |\la_{i'}-m_\eps| + R_{i',i}\Big( a|m_i-m_\eps| +|\la_i-(1-a)m_\eps|\Big).
    \]
We then set
       \begin{eqnarray}
    \label{9s.24}
C_{x,i,i'}(i',i') &=& R_{i',i'} + aR_{i',i}\frac {R_{i,i'}}{1-aR_{i,i}}\nn\\
C_{x,i,i'}(i',i) &=&  R_{i',i} \Big( a \frac{R_{i,i}(1-a)}{1-aR_{i,i}}
+ 1-a \Big)
    \end{eqnarray}
and
       \begin{eqnarray*}
\sum_{j=i,i'}C_{x,i,i'}(i',j) &=& 
R_{i',i'} +\frac {aR_{i',i}R_{i,i'}}{1-aR_{i,i}}
+R_{i',i}\Big(\frac{aR_{i,i}(1-a)}{1-aR_{i,i}} +(1-a)\Big)
\\&\le & R_{i',i'} + R_{i',i} -  a  R_{i',i}\Big(1- \frac{R_{i,i'}}{1-aR_{i,i}} -
\frac{R_{i,i}(1-a)}{1-aR_{i,i}} \Big)
\\&\le &  r  -  a  R_{i',i} \Big(1- \frac{R_{i,i'}+R_{i,i}(1-a)}{1-aR_{i,i}}  \Big)
<r
    \end{eqnarray*}
having bounded in the last bracket $R_{i,i'}+R_{i,i} <1$.

\vskip2cm

\end{appendix}

\noindent {\bf Acknowledgement}

MEV thanks the warm hospitality of GSSI, L'Aquila, where
part of this research was done.

Research partially supported by CNPq grant 474233/2012-0.
MEV's work is partially supported by CNPq grant 304217/2011-5 and
Faperj grant E-24/2013-132035.
LRF's work is partially supported by CNPq grant 305760/2010-6 and
Fapesp grant 2009/52379-8.

\end{document}